\documentclass{amsart}
\usepackage{amssymb,stmaryrd}

\newtheorem*{thma}{Theorem~A}
\newtheorem*{thmb}{Theorem~B}
\newtheorem*{thmc}{Theorem~C}
\newtheorem*{thmd}{Theorem~D}
\newtheorem{thm}{Theorem}[section]
\newtheorem{cor}[thm]{Corollary}

\newtheorem{prop}[thm]{Proposition}
\newtheorem{fact}[thm]{Fact}
\newtheorem{lemma}[thm]{Lemma}
\newtheorem{claim}{Claim}[thm]

\theoremstyle{defn}
\newtheorem{defn}[thm]{Definition}

\newtheorem{conv}[thm]{Convention}

\theoremstyle{remark}
\newtheorem{remark}[thm]{Remark}

\newcommand\last[2]{\eth_{#1,#2}}

\newcommand\sq{\sqsubseteq}
\newcommand\s{\subseteq}
\newcommand\br{\blacktriangleright}
\newcommand{\forces}{\Vdash}
\renewcommand{\restriction}{\mathbin\upharpoonright}
\renewcommand\mid{\mathrel{|}\allowbreak}
\newcommand\diagonal{\bigtriangleup}
\newcommand\diagonaal{\bigtriangledown}
\newcommand\symdiff{\triangle}
\DeclareMathOperator{\id}{id}
\DeclareMathOperator{\ns}{NS}
\DeclareMathOperator{\reg}{Reg}
\DeclareMathOperator{\card}{Card}
\DeclareMathOperator{\cf}{cf}
\DeclareMathOperator{\cl}{cl}
\DeclareMathOperator{\Tr}{Tr}
\DeclareMathOperator{\tr}{tr}
\DeclareMathOperator{\im}{Im}
\DeclareMathOperator{\otp}{otp}

\DeclareMathOperator{\acc}{acc}
\DeclareMathOperator{\nacc}{nacc}

\DeclareMathOperator{\pr}{Pr}
\DeclareMathOperator{\pl}{P\ell}
\DeclareMathOperator{\stat}{stat}
\DeclareMathOperator{\ssup}{ssup}

\author{Assaf Rinot}
\address{Department of Mathematics, Bar-Ilan University, Ramat-Gan 5290002, Israel.}
\urladdr{http://www.assafrinot.com}

\author{Jing Zhang}
\address{Department of Mathematics, Bar-Ilan University, Ramat-Gan 5290002, Israel.}
\urladdr{https://jingjzzhang.github.io/}


\title{Transformations of the transfinite plane}
\begin{document}
\begin{abstract} We study the existence of transformations of the transfinite plane that allow one to reduce Ramsey-theoretic statements concerning uncountable Abelian groups
into classical partition relations for uncountable cardinals.

To exemplify: we prove that for every inaccessible cardinal $\kappa$,
if $\kappa$ admits a stationary
set that does not reflect at inaccessibles, then the classical negative partition relation $\kappa\nrightarrow[\kappa]^2_\kappa$ implies that for every Abelian group $(G,+)$ of size $\kappa$, there exists a map $f:G\rightarrow G$
such that, for every  $X\s G$ of size $\kappa$ and every $g\in G$, 
there exist $x\neq y$ in $X$ such that $f(x+y)=g$. 
\end{abstract}
\maketitle
\section{Introduction}

Ramsey's theorem \cite{ramsey} asserts that every infinite graph contains an infinite induced subgraph which is either a clique or an anti-clique.
In other words, for every function (or \emph{coloring}, or \emph{partition}, depending on one's perspective) $c:[\mathbb N]^2\rightarrow 2$, there exists an infinite $X\s\mathbb N$ which is \emph{monochromatic} 
in the sense that, for some $i\in 2$, $c(x,y)=i$ for every pair $x<y$ of elements of $X$.
A strengthening of Ramsey's theorem due to Hindman \cite{MR0349574} concerns the additive structure $(\mathbb N,+)$ and asserts that for every partition $c:\mathbb N\rightarrow2$, there exists an infinite $X\s\mathbb N$ 
which is monochromatic in the  sense that, for some $i\in 2$, 
for every finite increasing sequence $x_0<\cdots<x_n$ of elements of $X$, $c(x_0+\cdots+x_n)=i$.

A natural generalization of Ramsey's and Hindman's theorems would assert that in any $2$-partition of an uncountable structure, there must exist an uncountable monochromatic subset.
However, this is not case. Already in the early 1930's, Sierpi\'nski found a coloring $c:[\mathbb R]^2\rightarrow2$ admitting no uncountable monochromatic set \cite{MR1556708}.
In contrast, a counterexample concerning the additive structure $(\mathbb R,+)$ was discovered only a few years ago \cite{MR3696151}, by Hindman, Leader and Strauss.

In this paper, we study the existence of transformations of the transfinite plane that allow one, among other things, to reduce the additive problem into the considerably simpler Ramsey-type problem.

Throughout the paper, $\kappa$ denotes a regular uncountable cardinal,
and $\theta,\chi$ denote (possibly finite) cardinals $\le\kappa$. The class of transformations of interest is captured by the following definition.

\begin{defn}\label{defpl1} $\pl_1(\kappa)$ asserts the existence of a transformation $\mathbf t:[\kappa]^2\rightarrow[\kappa]^2$ satisfying the following:
\begin{itemize}
\item for every $(\alpha,\beta)\in[\kappa]^2$, if $\mathbf t(\alpha,\beta)=(\alpha^*,\beta^*)$, then $\alpha^*\le\alpha<\beta^*\le\beta$;
\item for every family $\mathcal A$ consisting of $\kappa$ many pairwise disjoint finite subsets of $\kappa$,
there exists a stationary $S\s\kappa$ such that, for every pair $\alpha^*<\beta^*$ of elements of $S$,
there exists a pair $a<b$ of elements of $\mathcal A$ with $\mathbf t[a\times b]=\{(\alpha^*,\beta^*)\}$.
\end{itemize}
\end{defn}

\begin{thma} If $\pl_1(\kappa)$ holds,
then the following are equivalent:
\begin{itemize}
\item There exists a coloring $c:[\kappa]^2\rightarrow\theta$ such that, for every $X\s\kappa$ of size $\kappa$,
and every $\tau\in\theta$, there exist $x\neq y$ in $X$ such that $c(x,y)=\tau$;
\item For every Abelian group $(G,+)$ of size $\kappa$, there exists a coloring $c:G\rightarrow\theta$
such that, for all $X,Y\s G$ of size $\kappa$,
and every $\tau\in\theta$, there exist $x\in X$ and $y\in Y$ such that $c(x+y)=\tau$. 
\end{itemize}
\end{thma}

As the proof of Theorem~A will make clear, the theorem remains valid even after relaxing Definition~\ref{defpl1}
to omit the first bullet and to weaken ``stationary $S\s\kappa$'' into ``cofinal $S\s\kappa$''.
The reason we have added these extra requirements is to connect this line of investigation with other well-known problems,
such as the problem of whether the product of any two $\kappa$-cc posets must be $\kappa$-cc (cf.~\cite{paper18}):
\begin{thmb}  If $\pl_1(\kappa)$ holds, then for every positive integer $n$ there exists a poset $\mathbb P$ such that $\mathbb P^n$ satisfies the $\kappa$-cc, but $\mathbb P^{n+1}$ does not.
\end{thmb}

Now, to formulate the main results of this paper, let us consider a more informative variation of $\pl_1(\kappa)$.
\begin{defn}\label{fulldefpl1} $\pl_1(\kappa,\theta,\chi)$ asserts the existence of a function $\mathbf t:[\kappa]^2\rightarrow[\kappa]^3$ satisfying the following:
\begin{itemize}
\item for all $(\alpha,\beta)\in[\kappa]^2$, if $\mathbf t(\alpha,\beta)=(\tau^*,\alpha^*,\beta^*)$, then $\tau^*\le\alpha^*\le\alpha<\beta^*\le\beta$;
\item for all $\sigma<\chi$ and every family $\mathcal A\s [\kappa]^{\sigma}$
consisting of $\kappa$ many pairwise disjoint sets,
there exists a stationary $S\s\kappa$ such that, for all $(\alpha^*,\beta^*)\in[S]^2$ and $\tau^*<\min\{\theta,\alpha^*\}$,
there exist $(a,b)\in[\mathcal A]^2$ with $\mathbf t[a\times b]=\{(\tau^*,\alpha^*,\beta^*)\}$.
\end{itemize}
\end{defn}
\begin{remark} $\pl_1(\kappa)$ of Definition~\ref{defpl1} is $\pl_1(\kappa,1,\aleph_0)$.
\end{remark}

In \cite{paper13}, by building on the work of Eisworth in \cite{MR3087058,MR3087059}, 
the first author proved that $\pl_1(\lambda^+,\cf(\lambda),\cf(\lambda))$ holds for every singular cardinal $\lambda$.\footnote{The first bullet
of Definition~\ref{fulldefpl1} is not stated explicitly, but may be verified to hold in all the relevant arguments of \cite{MR3087058,MR3087059,paper13}.}
The proof of that theorem was a combination of walks on ordinals, club-guessing considerations, applications of elementary submodels, and oscillation of {\it pcf} scales.
Here, we replace the last ingredient by the oscillation oracle $\pl_6(\ldots)$ from \cite{paper15},
and there are a few additional differences which are too technical to state at this point.

The main result of this paper reads as follows:
\begin{thmc} For $\chi=\cf(\chi)\ge\omega$, $\pl_1(\kappa,\theta,\chi)$ holds in any of the following cases:
\begin{enumerate}
\item $\chi<\chi^+<\theta=\kappa$ and $\square(\kappa)$ holds;
\item $\chi<\chi^+<\theta=\kappa$ and $E^\kappa_{\ge\chi}$ admits a stationary set that does not reflect;
\item $\chi<\chi^+=\theta<\kappa$, $\kappa$ is inaccessible, and $E^\kappa_{\ge\chi}$ admits a stationary set that does not reflect at inaccessibles.
\end{enumerate}
\end{thmc}

By the results of Subsection~\ref{subsectionpr1} below, the principle $\pl_1(\kappa,\theta,\chi)$ is strictly stronger than Shelah's principle $\pr_1(\kappa,\kappa,\theta,\chi)$.
Thus, Clause~(1) improves the main result of \cite{paper18} and Clause~(2) improves the main result of \cite{paper15}.
The result of Clause~(3) provides, in particular, an affirmative answer to a question posed by Eisworth to the first author at the \emph{Set Theory} meeting in Oberwolfach, January 2014.

We conclude the introduction, mentioning two findings in the other direction.
\begin{thmd}For a strongly inaccessible cardinal $\kappa$: \begin{enumerate}
\item the existence of a coherent $\kappa$-Souslin tree does not imply $\pl_1(\kappa)$;
\item for any $\chi\in\reg(\kappa)$, the existence of a nonreflecting stationary subset of $E^\kappa_\chi$
does not imply $\pl_1(\kappa,1,\chi^+)$.
\end{enumerate}
\end{thmd}

\subsection{Organization of this paper} In Section~\ref{section1}, we establish some facts about walks on ordinals,
and present a connection between $\pl_1(\kappa,\ldots)$ and two other concepts: the coloring principle $\pr_1(\kappa,\ldots)$ 
and the $C$-sequence number, $\chi(\kappa)$.
The proofs of Theorems A, B and $D$ will be found there.

In Section~\ref{pl6section}, we prove that a strong form of the oscillation oracle $\pl_6(\nu^+,\nu)$ holds for any infinite regular cardinal $\nu$.
This fact will play a role in the later sections.

In Section~\ref{mainthm5}, we provide a proof of Clause~(2) of Theorem~C.
The proof is split into two cases:  $\kappa>\chi^{++}$ and $\kappa=\chi^{++}$.

In Section~\ref{Squaresection}, we provide a proof of Clause~(1) of Theorem~C. 

In Section~\ref{lastclause}, we provide a proof of Clause~(3) of Theorem~C. 

\subsection{Further results}
In an upcoming paper \cite{paper45}, we address the validity of the strongest possible instances of $\pl_1(\kappa,\theta,\chi)$.
Some of the main findings are:
\begin{itemize}
\item $\pl_1(\lambda^+,1,\lambda)$ fails for $\lambda$ singular,
so that Theorem~C is optimal whenever $\kappa$ is a successor of a singular cardinal;
\item $\pl_1(\lambda^+,1,\lambda^+)$ fails for $\lambda$ regular;
\item $\pl_1(\lambda^+,\lambda^+,\lambda)$ holds for $\lambda$ regular satisfying $2^\lambda=\lambda^+$;
\item $\pl_1(\aleph_1,\aleph_1,n)$ holds for all positive integers $n$;
\item $\pl_1(\kappa,\kappa,\kappa)$ holds for $\kappa$ inaccessible such that $\square(\kappa)$ and $\diamondsuit^*(\kappa)$ both hold.
\end{itemize}

\subsection{Notation and conventions}
Let $E^\kappa_\chi:=\{\alpha < \kappa \mid \cf(\alpha) = \chi\}$,
and define $E^\kappa_{\le \chi}$, $E^\kappa_{<\chi}$, $E^\kappa_{\ge \chi}$, $E^\kappa_{>\chi}$,  $E^\kappa_{\neq\chi}$ analogously.
For an ideal $\mathcal I$ over $\kappa$, we write $\mathcal I^+:=\mathcal P(\kappa)\setminus\mathcal I$.
The collection of all sets of hereditary cardinality less than $\kappa$ is denoted by $\mathcal H_\kappa$.
The set of all infinite (resp.~infinite and regular) cardinals below $\kappa$
is denoted by $\card(\kappa)$ (resp.~$\reg(\kappa)$).
The length of a finite sequence $\varrho$ is denoted by $\ell(\varrho)$.
For a subset $S\s\kappa$, we let $\Tr(S):=\{\alpha\in E^\kappa_{>\omega}\mid S\cap\alpha\text{ is stationary in }\alpha\}$;
we say that $S$ is \emph{nonreflecting} (resp.\ \emph{nonreflecting at inaccessibles}) iff $\Tr(S)$ is empty (resp.\ contains no inaccessible cardinals).
For a set of ordinals $a$, we write 
$\ssup(a):=\sup\{\alpha+1\mid \alpha\in a\}$,
$\acc^+(a) := \{\alpha < \ssup(a) \mid \sup(a \cap \alpha) = \alpha > 0\}$,
$\acc(a) := a \cap \acc^+(a)$, $\nacc(a) := a \setminus \acc(a)$,
and $\cl(a):= a\cup\acc^+(a)$.
For sets of ordinals, $a$ and $b$, we let $a\circledast b:=\{(\alpha,\beta)\in a\times b\mid \alpha<\beta\}$,
and write $a < b$ to express that $a\times b$ coincides with $a\circledast b$.

For any set $\mathcal A$, we write
$[\mathcal A]^\chi:=\{ \mathcal B\s\mathcal A\mid |\mathcal B|=\chi\}$ and
$[\mathcal A]^{<\chi}:=\{\mathcal B\s\mathcal A\mid |\mathcal B|<\chi\}$.
This convention admits two refined exceptions:
\begin{itemize}
\item For an ordinal $\sigma$ and a set of ordinals $A$, we write 
$[A]^\sigma$ for $\{ B\s A\mid \otp(B)=\sigma\}$;
\item For a set $\mathcal{A}$ which is either an ordinal or a collection of sets
of ordinals, we interpret $[\mathcal{A}]^2$ as the collection of \emph{ordered} pairs $\{ (a,b)\in\mathcal A\times\mathcal A\mid a<b\}$.
\end{itemize}
In particular, $[\kappa]^2=\{(\alpha,\beta)\mid \alpha<\beta<\kappa\}$.
Likewise, we let $[\kappa]^3:=\{(\alpha,\beta,\gamma)\in\kappa\times\kappa\times\kappa\mid \alpha<\beta<\gamma<\kappa\}$.

\section{Warming up}\label{section1}

\subsection{The foundations of walks on ordinals}\label{subsectionwalks}
\begin{defn}[folklore] $\kappa\nrightarrow[\kappa]^2_\theta$ 
(resp.~$\kappa\nrightarrow[\stat]^2_\theta$)
asserts the existence of a coloring $c:[\kappa]^2\rightarrow\theta$ such that, for every cofinal (resp.~stationary) $X\s\kappa$,
and every $\tau\in\theta$, there exist $(x,y)\in[X]^2$ such that $c(x,y)=\tau$.

Likewise, $\kappa\nrightarrow[\kappa;\kappa]^2_\theta$ 
(resp.~$\kappa\nrightarrow[\stat;\stat]^2_\theta$)
asserts the existence of a coloring $c:[\kappa]^2\rightarrow\theta$ such that, for every two cofinal (resp.~stationary) $X,Y\s\kappa$,
and every $\tau\in\theta$, there exist $(x,y)\in X\circledast Y$ such that $c(x,y)=\tau$.
\end{defn}

In an unpublished note from 1981, Todorcevic proved that $\omega_1\nrightarrow[\stat;\stat]^2_{\omega_1}$ holds.
A few years later, in \cite{TodActa},
the method of \emph{walks on ordinals} was introduced, with the following theorem serving as the primary application.
\begin{fact}[Todorcevic, \cite{TodActa}]\label{nonreftm} $\omega_1\nrightarrow[\omega_1]^2_{\omega_1}$ holds. 
Furthermore, for every regular uncountable cardinal $\kappa$ admitting a nonreflecting stationary set,
$\kappa\nrightarrow[\kappa]^2_{\kappa}$ holds.
\end{fact}

Later, by a series of results of Shelah concerning cardinals $\kappa>\aleph_1$ together with a result of Moore concerning $\kappa=\aleph_1$, $\kappa\nrightarrow[\kappa;\kappa]^2_\kappa$ holds for any cardinal $\kappa$ which is the successor of an infinite regular cardinal;
see \cite{paper14} for an historical account and a uniform proof of the following:
\begin{fact}[Shelah, Moore]\label{rectangular} $\nu^+\nrightarrow[\nu^+;\nu^+]^2_{\nu^+}$ holds for any infinite regular cardinal $\nu$.
\end{fact}

In this subsection, we present a few basic components of the theory of walks on ordinals,
which we will be using throughout the rest of the paper.

\begin{defn} 
For a set of ordinals $\Gamma$, a \emph{$C$-sequence over $\Gamma$}
is a sequence of sets $\langle C_\alpha\mid\alpha\in\Gamma\rangle$
such that, for all $\alpha\in\Gamma$, $C_\alpha$ is a closed subset of $\alpha$ with $\sup(C_\alpha)=\sup(\alpha)$.
\end{defn}

For the rest of this subsection, 
let us fix a $C$-sequence $\vec C=\langle C_\alpha\mid\alpha<\kappa\rangle$ over $\kappa$.

\begin{defn}[Todorcevic, \cite{TodActa}] From $\vec C$, we derive maps $\Tr:[\kappa]^2\rightarrow{}^\omega\kappa$,
$\rho_2:[\kappa]^2\rightarrow	\omega$,
$\tr:[\kappa]^2\rightarrow{}^{<\omega}\kappa$ and $\lambda:[\kappa]^2\rightarrow\kappa$, as follows.
Let $(\alpha,\beta)\in[\kappa]^2$ be arbitrary.
\begin{itemize}
\item $\Tr(\alpha,\beta):\omega\rightarrow\kappa$ is defined by recursion on $n<\omega$:
$$\Tr(\alpha,\beta)(n):=\begin{cases}
\beta,&n=0\\
\min(C_{\Tr(\alpha,\beta)(n-1)}\setminus\alpha),&n>0\ \&\ \Tr(\alpha,\beta)(n-1)>\alpha\\
\alpha,&\text{otherwise}
\end{cases}
$$
\item $\rho_2(\alpha,\beta):=\min\{n<\omega\mid \Tr(\alpha,\beta)(n)=\alpha\}$;
\item $\tr(\alpha,\beta):=\Tr(\alpha,\beta)\restriction \rho_2(\alpha,\beta)$;
\item $\lambda(\alpha,\beta):=\max\{ \sup(C_{\Tr(\alpha,\beta)(i)}\cap\alpha) \mid i<\rho_2(\alpha,\beta)\}$.
\end{itemize}
\end{defn}

The next two facts are quite elementary. They are reproduced with proofs as Claims 3.1.1 and 3.1.2 of \cite{paper15}.
\begin{fact}\label{fact1} Whenever $0<\beta<\gamma<\kappa$, if $\beta\notin\bigcup_{\alpha<\kappa}\acc(C_\alpha)$, then $\lambda(\beta,\gamma)<\beta$.
\end{fact}
\begin{fact}\label{fact2} Whenever $\lambda(\beta,\gamma)<\alpha<\beta<\gamma<\kappa$, 
$\tr(\alpha,\gamma)=\tr(\beta,\gamma){}^\smallfrown \tr(\alpha,\beta)$.
\end{fact}

\begin{conv}\label{conv28} For any coloring $f:[\kappa]^2\rightarrow\kappa$ and $\delta<\kappa$,
while $(\delta,\delta)\notin[\kappa]^2$, we extend the definition of $f$, and agree to let $f(\delta,\delta):=0$.
\end{conv}

\begin{lemma}\label{fact3} 
Let $(\alpha,\gamma)\in[\kappa]^2$. For every $\beta\in\im(\tr(\alpha,\gamma))$,
$$\lambda(\alpha,\gamma)=\max\{\lambda(\beta,\gamma),\lambda(\alpha,\beta)\}.$$
\end{lemma}
\begin{proof} Let $\beta$ be as above, so that $\tr(\alpha,\gamma)=\tr(\beta,\gamma){}^\smallfrown\tr(\gamma,\beta)$. We have 
\begin{align*}
\lambda(\alpha,\gamma)=&\max\{ \sup(C_{\tau}\cap\alpha) \mid \tau\in\im(\tr(\alpha,\gamma))\}=\\
&\max\{ \sup(C_{\tau_0}\cap\alpha),\sup(C_{\tau_1}\cap\alpha)  \mid \tau_0\in\im(\tr(\beta,\gamma)),\tau_1\in\im(\tr(\alpha,\beta))\}\le\\
&\max\{ \sup(C_{\tau_0}\cap\beta),\sup(C_{\tau_1}\cap\alpha)  \mid \tau_0\in\im(\tr(\beta,\gamma)),\tau_1\in\im(\tr(\alpha,\beta))\}=\\
&\max\{\lambda(\beta,\gamma),\lambda(\alpha,\beta)\},
\end{align*}
and
\begin{align*}
\lambda(\alpha,\gamma)=&
\max\{ \sup(C_{\tau_0}\cap\alpha),\sup(C_{\tau_1}\cap\alpha)  \mid \tau_0\in\im(\tr(\beta,\gamma)),\tau_1\in\im(\tr(\alpha,\beta))\}\ge\\
&\max\{ \sup(C_{\tau_1}\cap\alpha)  \mid \tau_1\in\im(\tr(\alpha,\beta))\}=\lambda(\alpha,\beta).
\end{align*}

So, if $\lambda(\alpha,\gamma)\neq\max\{\lambda(\beta,\gamma),\lambda(\alpha,\beta)\}$,
then $\lambda(\alpha,\gamma)<\lambda(\beta,\gamma)$,
and we may fix the least $i<\rho_2(\beta,\gamma)$ to satisfy $\sup(C_{\Tr(\beta,\gamma)(i)}\cap\alpha)<\sup(C_{\Tr(\beta,\gamma)(i)}\cap\beta)$;
but then $\Tr(\alpha,\gamma)(i+1)=\min(C_{\Tr(\beta,\gamma)(i)}\setminus\alpha)<\beta\le\Tr(\beta,\gamma)(i+1)$,
contradicting the fact that $\tr(\beta,\gamma){}^\smallfrown\langle\beta\rangle\sq\tr(\alpha,\gamma)$.
\end{proof}

\begin{defn} For every $(\alpha,\beta)\in[\kappa]^2$, we define an ordinal $\last{\alpha}{\beta}\in[\alpha,\beta]$ via:
$$\last{\alpha}{\beta}:=\begin{cases}
\alpha,&\text{if }\lambda(\alpha,\beta)<\alpha;\\
\min(\im(\tr(\alpha,\beta)),&\text{otherwise}.
\end{cases}$$
\end{defn}

\begin{lemma}\label{last} Let $(\alpha,\beta)\in[\kappa]^2$ with $\alpha>0$. Then
\begin{enumerate}
\item $\lambda(\last{\alpha}{\beta},\beta)<\alpha$;\footnote{Recall Convention~\ref{conv28}.}
\item If $\last{\alpha}{\beta}\neq\alpha$, then $\alpha\in\acc(C_{\last{\alpha}{\beta}})$;
\item $\tr(\last{\alpha}{\beta},\beta)\sq\tr(\alpha,\beta)$.
\end{enumerate}
\end{lemma}
\begin{proof} To avoid trivialities, assume that $\lambda(\alpha,\beta)=\alpha$.
Let $\beta_0>\cdots>\beta_n>\beta_{n+1}$ denote the decreasing enumeration of the elements of $\im(\Tr(\alpha,\beta))$,
so that $\beta_0=\beta$, $\beta_n=\last{\alpha}{\beta}$, and $\beta_{n+1}=\alpha$.
For each $i<n$, $C_{\beta_i}\cap [\alpha,\beta_{i+1})$ is empty,
so that $\min(C_{\beta_i}\setminus\beta_n)=\min(C_{\beta_i}\setminus\alpha)$
and $\sup(C_{\beta_i}\cap\beta_n)=\sup(C_{\beta_i}\cap\alpha)<\alpha$.
Now, the three clauses follow immediately.
\end{proof}

For the purpose of this paper, we also introduce the following ad-hoc notation.
\begin{defn}\label{etanotation}
For every ordinal $\eta<\kappa$ and a pair $(\alpha,\beta)\in[\kappa]^2$, we let
$$\eta_{\alpha,\beta}:=\min\{ n<\omega\mid \eta\in C_{\Tr(\alpha,\beta)(n)}\text{ or }n=\rho_2(\alpha,\beta)\}+1.$$
\end{defn}

\subsection{Relationship to the $C$-sequence number}
\begin{defn}[The $C$-sequence number of $\kappa$, \cite{paper35}] \label{c_seq_num_def}
  If $\kappa$ is weakly compact, then we define $\chi(\kappa):=0$. Otherwise, we let
  $\chi(\kappa)$ denote the least (finite or infinite) cardinal $\chi\le\kappa$
  such that, for every $C$-sequence $\langle C_\beta\mid\beta<\kappa\rangle$,
  there exist $\Delta\in[\kappa]^\kappa$ and $b:\kappa\rightarrow[\kappa]^{\chi}$
  with $\Delta\cap\alpha\s\bigcup_{\beta\in b(\alpha)}C_\beta$
  for every $\alpha<\kappa$.
\end{defn}

\begin{fact}[Todorcevic, {\cite[Theorem~8.1.11]{TodWalks}}]\label{oscfact} If $\chi(\kappa)>1$, then $\kappa\nrightarrow[\kappa]^2_\omega$.
\end{fact}
\begin{fact}[Lambie-Hanson and Rinot, \cite{paper35}]\label{fact29} If $\chi(\kappa)\le1$, then $\kappa$ is (in fact, greatly) Mahlo and for 
every $C$-sequence $\langle C_\beta\mid \beta\in\reg(\kappa)\rangle$ over $\reg(\kappa)$,
there exists a club $D\s\kappa$ satisfying the following.
For every $\alpha<\kappa$, there exists $\beta\in\reg(\kappa)$, such that $D\cap\alpha\s C_\beta$.
\end{fact}

\begin{lemma}\label{lemma212} If $\chi(\kappa)\le 1$, then $\pl_1(\kappa,1,2)$ fails.
\end{lemma}
\begin{proof} Suppose that $\pl_1(\kappa,1,2)$ holds.
\begin{claim} There exists a function $s:[\kappa]^2\rightarrow\kappa$ satisfying the following:
\begin{enumerate}
\item for all $(\alpha,\beta)\in\kappa\circledast\acc(\kappa)$, $\alpha<s(\alpha,\beta)<\beta$;
\item for every cofinal $A\s\kappa$, $s``[A]^2$ is stationary.
\end{enumerate}
\end{claim}
\begin{proof}
Fix $\mathbf t:[\kappa]^2\rightarrow[\kappa]^3$ witnessing $\pl_1(\kappa,1,2)$.
Define $s:[\kappa]^2\rightarrow\kappa$ by letting $s(\alpha,\beta):=\beta^*$ 
whenever $\mathbf t(\alpha,\beta)=(\tau^*,\alpha^*,\beta^*)$ with $\alpha<\beta^*<\beta$,
and letting $s(\alpha,\beta):=\alpha+1$, otherwise.
To verify Clause~(2), let $A$ be an arbitrary cofinal subset of $\kappa$.
Set $C:=\acc^+(A)$ and $A':=A\setminus C$, so that $A'$ is a discrete cofinal subset of $A$.
As $\{ \{\alpha\}\mid \alpha\in A'\}$ is a subset of $[\kappa]^1$ consisting of $\kappa$ many pairwise disjoint sets,
we may now fix a stationary $S\s\kappa$ such that, for all $(\alpha^*,\beta^*)\in[S]^2$,
there exists $(\alpha,\beta)\in[A']^2$ with $\mathbf t(\alpha,\beta)=(0,\alpha^*,\beta^*)$.
We claim that $s``[A]^2$ covers the stationary set $S^*:=(S\cap C)\setminus\{\min(S)\}$.

To see this, let $\beta^*\in S^*$ be arbitrary. 
Put $\alpha^*:=\min(S)$.
Fix $(\alpha,\beta)\in[A']^2$ such that $\mathbf t(\alpha,\beta)=(0,\alpha^*,\beta^*)$.
We know that $\alpha^*\le\alpha<\beta^*\le\beta$
and that $\beta^*\in C$ while $\beta\in A\setminus C$. So $\alpha<\beta^*<\beta$,
and hence $s(\alpha,\beta)=\beta^*$, as sought.
\end{proof}

Suppose that $\chi(\kappa)\le1$, and yet there exists a 
function $s:[\kappa]^2\rightarrow \kappa$ as in the preceding claim.
Set $C_\omega:=\omega$. For any uncountable $\beta\in\reg(\kappa)$, let $$C_\beta:=\{\gamma<\beta\mid \forall \alpha<\gamma [ s(\alpha, \beta) < \gamma]\}$$
be the club of closure points of the function $s(\cdot, \beta)$.
Note that, for any $\alpha<\beta$, $s(\alpha,\beta)\notin C_\beta$, since $\alpha<s(\alpha,\beta)$.

Now, by Fact~\ref{fact29},  we may fix a club $D\s\kappa$ with the property that, for every $\alpha<\kappa$, there exists $\beta\in\reg(\kappa)$ with $D\cap\alpha\s C_\beta$.

Recursively build a (discrete) subset $A\s(\{0\}\cup(\reg(\kappa)\setminus\omega_1))$ such that, for any nonzero $\beta\in A$, $\beta^-:=\sup(A\cap \beta)$ is smaller than $\beta$, and  $D\cap(\beta^-+1)\s C_\beta$. 
Then, let $E$ be the closure of $\bigcup\{C_\beta\setminus\beta^- \mid  \beta\in A, \beta\neq0\}$ in $\kappa$,
and note that, for every $\beta\in A$, $E\cap(\beta^-,\beta)=C_\beta\cap(\beta^-,\beta)$.

As $A$ is cofinal, $S:=s``[A]^2$ is stationary,
so that we may pick $\beta^*\in S\cap D\cap E$.
Fix a pair $(\alpha,\beta)\in[A]^2$ with $s(\alpha,\beta)=\beta^*$.

\begin{claim}$\beta^*\in C_\beta$.
\end{claim}
\begin{proof} As $(\alpha,\beta)\in[A]^2$, we we know that $\beta$ is a regular uncountable cardinal.
So, by the hypothesis on $s$, $\alpha<\beta^*<\beta$.
Now, there are two cases to consider:

$\br$ If  $\beta^*\le \beta^-$, then $\beta^*\in D\cap(\beta^-+1)\s C_\beta$.

$\br$ Otherwise, $\beta^-<\beta^*<\beta$, so that  $\beta^*\in E\cap(\beta^-,\beta)= C_\beta\cap(\beta^-,\beta)$.
\end{proof}

However, we have observed earlier that $s(\alpha,\beta)\notin C_\beta$,
meaning that $\beta^*\notin C_\beta$. This contradicts the preceding claim.
\end{proof}

\subsection{Relationship to Shelah's principle $\pr_1$}\label{subsectionpr1}

\begin{defn}[Shelah, \cite{shelah_productivity}]\label{def_pr1}
$\pr_1(\kappa, \kappa, \theta, \chi)$ asserts the existence of a coloring $c:[\kappa]^2 \rightarrow \theta$ such that for every  $\sigma<\chi$, every family 
$\mathcal{A}  \subseteq [\kappa]^{\sigma}$ consisting of $\kappa$ many pairwise disjoint
  sets, and every $i < \theta$, there is $(a,b) \in [\mathcal{A}]^2$ such
  that $c[a \times b] = \{i\}$.
\end{defn}

Note that $\pr_1(\kappa,\kappa,\theta,2)$ is equivalent to $\kappa\nrightarrow[\kappa]^2_\theta$.

\begin{lemma}\label{thm214} Any of the following implies that $\pr_1(\kappa,\kappa,\theta,\chi)$ holds:
\begin{enumerate}
\item $\pl_1(\kappa,\theta,\chi)$;
\item $\pl_1(\kappa,1,\chi)$ and $\kappa\nrightarrow[\stat(\kappa)]^2_\theta$;
\item $\pl_1(\kappa,\cf(\theta),\chi)$ and $\kappa\nrightarrow[\stat(\kappa)]^2_\eta$ for all $\eta<\theta$;
\item $\pl_1(\kappa,\nu,\chi)$ and there exists a $\nu^+$-cc poset $\mathbb P$ such that $\Vdash_{\mathbb P}\kappa\nrightarrow[\kappa]^2_\theta$.
\end{enumerate}
\end{lemma}
\begin{proof} 
(1) Let $\mathbf t:[\kappa]^2\rightarrow[\kappa]^3$  be a witness to $\pl_1(\kappa,\theta,\chi)$.
Define $c^*:[\kappa]^2\rightarrow\theta$ via $c^*(\alpha,\beta):=\tau^*$ whenever $\mathbf t(\alpha,\beta)=(\tau^*,\alpha^*,\beta^*)$.
Then $c^*$ witnesses $\pr_1(\kappa,\kappa,\theta,\chi)$.

(2) Let $\mathbf t:[\kappa]^2\rightarrow[\kappa]^3$ be a witness to $\pl_1(\kappa,1,\chi)$,
and let $c:[\kappa]^2\rightarrow\theta$ be a witness to $\kappa\nrightarrow[\stat(\kappa)]^2_\theta$.
Define $c^*:[\kappa]^2\rightarrow\theta$ via $c^*(\alpha,\beta):=c(\alpha^*,\beta^*)$ whenever $\mathbf t(\alpha,\beta)=(\tau^*,\alpha^*,\beta^*)$.
Then $c^*$ witnesses $\pr_1(\kappa,\kappa,\theta,\chi)$.

(3) Let $\mathbf t:[\kappa]^2\rightarrow[\kappa]^3$ be a witness to $\pl_1(\kappa,\cf(\theta),\chi)$.
By Clause~(1), we may assume that $\theta$ is singular.
Thus, let $\langle \eta_i\mid i<\cf(\theta)\rangle$ be an increasing sequence of cardinals, converging to $\theta$.
For each $i<\cf(\theta)$, let $c_i:[\kappa]^2\rightarrow\eta_i$ be a witness to $\kappa\nrightarrow[\stat(\kappa)]^2_{\eta_i}$.
Define $c^*:[\kappa]^2\rightarrow\theta$ via $c^*(\alpha,\beta):=c_i(\alpha^*,\beta^*)$ whenever $\mathbf t(\alpha,\beta)=(i,\alpha^*,\beta^*)$.
Then $c^*$ witnesses $\pr_1(\kappa,\kappa,\theta,\chi)$.

(4) By Clause~(1), we may assume that $\nu<\theta$. Let $\mathbf t:[\kappa]^2\rightarrow[\kappa]^3$  be a witness to $\pl_1(\kappa,\nu,\chi)$.
Suppose that $\mathbb P$ is a $\nu^+$-cc poset such that $\Vdash_{\mathbb P}\kappa\nrightarrow[\kappa]^2_\theta$.
Fix a $\mathbb P$-name $\dot c$ for a coloring witnessing $\kappa\nrightarrow[\kappa]^2_\theta$ in the forcing extension by $\mathbb P$.
Define $d:[\kappa]^2\rightarrow\mathcal P(\theta)$ via $$d(\alpha,\beta):=\{\tau<\theta\mid \exists p(p\forces_{\mathbb P}\dot c(\check\alpha,\check\beta)=\check\tau)\}.$$ 
As $\mathbb P$ is $\nu^+$-cc, $|d(\alpha,\beta)|\le\nu$ for every $(\alpha,\beta)\in[\kappa]^2$, so that, we may define a function $e:[\kappa]^3\rightarrow\theta$ such that, all $(\alpha,\beta)\in[\kappa\setminus\nu]^2$, $d(\alpha,\beta)\s \{ e(i,\alpha,\beta)\mid i<\tau\}$.
It follows that $e\circ\mathbf t$ witnesses $\pr_1(\kappa,\kappa,\theta,\chi)$.
\end{proof}

We now establish Theorem~D.

\begin{prop}\label{prop219} Suppose that $\kappa$ is weakly compact and $\chi\in\reg(\kappa)$.
\begin{enumerate}
\item There exists a cofinality-preserving forcing extension in which $\kappa$ is strongly inaccessible,
there exists a coherent $\kappa$-Souslin tree, $\pr_1(\kappa,\kappa,\kappa,\omega)$ holds, 
yet $\pl_1(\kappa)$ fails.
\item There exists a cofinality-preserving forcing extension in which $\kappa$ is strongly inaccessible,
there exists a nonreflecting stationary subset of $E^\kappa_\chi$, yet $\pl_1(\kappa,1,\allowbreak\chi^+)$ fails.
\end{enumerate}
\end{prop}
\begin{proof} (1) In \cite[\S3]{paper35}, a cofinality-preserving forcing extension given by Kunen is revisited,
in which $\kappa$ remains strongly inaccessible and there exists a coherent $\kappa$-Souslin tree,
so that  $\pr_1(\kappa,\kappa,\kappa,\omega)$ holds. 
It is shown there that $\chi(\kappa)=1$ holds in this model, so that, by Lemma~\ref{lemma212}, $\pl_1(\kappa)$ fails.

(2) In \cite[\S3]{paper35}, the authors give a cofinality-preserving forcing extension in which there exists a 
nonreflecting stationary subset of $E^\kappa_\chi$, and $\pr_1(\kappa,\kappa,\kappa,\chi^+)$ fails.
By Fact~\ref{nonreftm} and Lemma~\ref{thm214}, $\pl_1(\kappa,1,\chi^+)$ must fail in this model.
\end{proof}

Next, we turn to derive Theorem~A:

\begin{cor} Suppose that $\pl_1(\kappa)$ holds.
For every cardinal $\theta\le\kappa$, the following are equivalent:
\begin{enumerate}
\item $\kappa\nrightarrow[\kappa]^2_\theta$;
\item $\kappa\nrightarrow[\kappa;\kappa]^2_\theta$;
\item $\pr_1(\kappa,\kappa,\theta,\omega)$;
\item For every Abelian group $(G,+)$ of size $\kappa$, there exists a coloring $d:G\rightarrow\theta$
such that, for all $X,Y\s G$ of size $\kappa$,
and every $\tau\in\theta$, there exist $x\in X$ and $y\in Y$ such that $d(x+y)=\tau$. 
\end{enumerate}
\end{cor}
\begin{proof}
$(3)\implies(2)\implies(1)$: This is trivial.

$(1)\implies(3)$: By Lemma~\ref{thm214}(2).

$(3)\implies(4)$: By Lemma~3.4 and \cite[Theorem~4.2]{paper27}.

$(4)\implies(1)$: As $([\kappa]^{<\omega},\symdiff)$ is an Abelian group of size $\kappa$,
let us fix a coloring $d:[\kappa]^{<\omega}\rightarrow\theta$ as in Clause~(4). Now define a coloring $c:[\kappa]^2\rightarrow\theta$
by stipulating $c(x,y):=d(\{x,y\})$. Clearly, $c$ witnesses that $\kappa\nrightarrow[\kappa]^2_\theta$ holds.
\end{proof}
\begin{remark} Compare the preceding with Conjecture~2 of \cite{paper18}.
\end{remark}

\begin{cor}\label{cor216} If $\pl_1(\kappa,1,\chi)$ holds, then so does $\pr_1(\kappa,\kappa,\omega,\chi)$.
\end{cor}
\begin{proof} To avoid trivialities, suppose that $\chi\ge 2$.
Then, by Lemma~\ref{lemma212}, $\chi(\kappa)>1$. 
Finally, by Fact~\ref{oscfact} and Theorem~\ref{thm214}(2), $\pr_1(\kappa,\kappa,\omega,\chi)$ holds.
\end{proof}

We are now ready to derive Theorem~B:

\begin{cor} Suppose that $\pl_1(\kappa)$ holds and $n$ is some positive integer.
Then there exists a poset $\mathbb P$ such that $\mathbb P^n$ satisfies the $\kappa$-cc, but $\mathbb P^{n+1}$ does not.
\end{cor}
\begin{proof} By Corollary~\ref{cor216}, in particular,
we may fix a coloring $c:[\kappa]^2\rightarrow n+1$ witnessing $\pr_1(\kappa,\kappa,n+1,\omega)$.
We define a poset $\mathbb P:=(P,{\le})$ by letting
$$P:=\{ (i,x)\mid i<n+1, x\in[\kappa]^{<\omega}, i\notin c``[x]^2\},$$
and letting $(i,x)\le (j,y)$ iff $i=j$ and $x\supseteq y$.
A moment's reflection makes it clear that 
$\{ \langle(i,\{\alpha\})\mid i<n+1\rangle\mid \alpha<\kappa\rangle$
forms a $\kappa$-sized antichain in $\mathbb P^{n+1}$.

We are left with showing that $\mathbb P^n$ does satisfy the $\kappa$-cc.
To this end, let $A$ be an arbitrary $\kappa$-sized subset of $\mathbb P^n$.
For every $p\in A$, write $p$ as $\langle (i^p_j,x^p_j)\mid j<n\rangle$.
By the pigeonhole principle, we may assume the existence of a sequence $\langle i_j\mid j<n\rangle$ 
such that, for every $p\in A$, $\langle i_j^p\mid j<n\rangle=\langle i_j\mid j<n\rangle$.
Find $i^*<n+1$ such that $i^*\neq i_j$ for all $j<n$.
By the $\Delta$-system lemma, we may also assume that, for every $j<n$, $\{ x^p_j\mid p\in A\}$ forms 
a $\Delta$-system with some room $r_j$. Let $r:=\bigcup_{j<n}r_j$.
Note that $r$ is finite (possibly empty).
By further thinning out we may assume that, for all $p\in A$ and $j<n$, $\min(x_j^p\setminus r_j)>\sup(r)$.
By one last step of thinning out, we may finally 
secure that $\{ \bigcup_{j<n}x_j^p\setminus r\mid p\in A\}$ form a family of $\kappa$-many pairwise disjoint finite sets.
Now, the choice of $c$ entails that we may find $p\neq q$ in $A$
such that:
\begin{enumerate}
\item $\max(\bigcup_{j<n}x_j^p\setminus r)<\min(\bigcup_{j<n}x_j^q\setminus r)$, and 
\item $c\left[(\bigcup\nolimits_{j<n}x_j^p\setminus r)\times (\bigcup\nolimits_{j<n}x_j^q\setminus r)\right]=\{i^*\}$.
\end{enumerate}
To see that $p$ and $q$ are compatible, fix arbitrary $j<n$ and $(\alpha,\beta)\in [x^p_j\cup x^q_j]^2$; we need to verify that $c(\alpha,\beta)\neq i_j$.
There are three possible options:

$\br$ If $(\alpha,\beta)\in [x^p_j]^2\cup[x^q_j]^2$, then since $i_j^p=i_j=i^q_j$, $c(\alpha,\beta)\neq i_j$.

$\br$ If $\alpha\in x_j^p\setminus x_j^q$ and $\beta\in x_j^q\setminus x_j^p$,
then $\alpha\in x_j^p\setminus r_j$ and $\beta\in x_j^q\setminus r_j$,
so that altogether $\alpha\in x_j^p\setminus r$ and $\beta\in x_j^q\setminus r$.
by Clause~(2), then, $c(\alpha,\beta)=i^*$.
In particular,  $c(\alpha,\beta)\neq i_j$.

$\br$ If $\alpha\in x_j^q\setminus x_j^p$ and $\beta\in x_j^p\setminus x_j^q$,
then $\alpha\in x_j^q\setminus r$ and $\beta\in x_j^p\setminus r$,
contradicting Clause~(1). So this case does not exist.
\end{proof}

\section{Improved oscillation}\label{pl6section}

In \cite{paper15}, the first author introduced the following oscillation principle:
\begin{defn}\label{defpl6} $\pl_6(\mu,\nu)$ asserts the existence of a map $d:{}^{<\omega}\mu\rightarrow\omega$
such that for every sequence $\langle (u_\alpha,v_\alpha,\sigma_\alpha)\mid\alpha<\mu\rangle$ and a function $\varphi:\mu\rightarrow\mu$ satisfying:
\begin{enumerate}
\item $\varphi$ is eventually regressive. That is, $\varphi(\alpha)<\alpha$ for co-boundedly many $\alpha<\mu$;
\item  $u_\alpha$ and $v_\alpha$ are nonempty elements of $[{}^{<\omega}\mu]^{<\nu}$;
\item $\alpha\in\im(\varrho)$ for all $\varrho\in u_\alpha$;
\item $\sigma_\alpha{}^\frown\langle\alpha\rangle\sqsubseteq \sigma$ for all $\sigma\in v_\alpha$,
\end{enumerate}
there exist $(\alpha,\beta)\in[\mu]^2$ with $\varphi(\alpha)=\varphi(\beta)$ such that, for all $\varrho\in u_\alpha$ and $\sigma\in v_\beta$, $d(\varrho{}^\frown\sigma)=\ell(\varrho)$ .
\end{defn}

The main result of \cite[\S2]{paper15} states that $\pl_6(\nu^+,\nu)$ holds for every infinite regular cardinal $\nu$.
In \cite{paper45}, we show that $\pl_6(\nu^+,\nu)$ fails for every singular cardinal $\nu$,
and that $\pl_6(\mu,\mu)$ fails for every infinite cardinal $\mu$.

In this paper, we shall be making use of two variations of $\pl_6(\nu^+,\nu)$.
The first variation reads as follows:

\begin{fact}\label{pl6} 
Suppose that $\mu=\nu^+$ for an infinite regular cardinal $\nu$.
Then there exists a map $d:{}^{<\omega}\mu\rightarrow\omega\times\mu\times\mu\times\mu$
such that, for every $\gamma^*<\mu$, and every sequence $\langle (u_\alpha,v_\alpha,\sigma_\alpha)\mid\alpha<\mu\rangle$
satisfying clauses (2)--(4) of Definition~\ref{defpl6},
there exist $(\alpha,\beta)\in [\mu]^2$
such that, for all $\varrho\in u_\alpha$ and $\sigma\in v_\beta$,
$d(\varrho{}^\smallfrown\sigma)=(\ell(\varrho),\alpha,\beta,\gamma^*)$.
\end{fact}
\begin{proof} This follows immediately from Theorems 2.3 and 2.6 of \cite{paper15}.
\end{proof}

The second variation reads as follows:

\begin{lemma}\label{lemma42} Suppose that $\mu=\nu^+$ for an infinite regular cardinal $\nu$.
Then there exist a map $d_0:{}^{<\omega}\mu\rightarrow\omega$
and a $\mu$-additive normal ideal $J$ on $\mu$ with $E^\mu_\nu\notin J$ such that,
for every sequence $\langle (u_\alpha,v_\alpha,\sigma_\alpha)\mid \alpha\in A\rangle$ with $A\in  J^+$
satisfying clauses (2)--(4) of Definition~\ref{defpl6},
there exist $(\alpha,\beta)\in[A]^2$
such that, for all $\varrho\in u_\alpha$ and $\sigma\in v_j$,
$d_0(\varrho{}^\smallfrown\sigma)=\ell(\varrho)$.
\end{lemma}
\begin{proof} By \cite[Theorem~2.6]{paper15}, we may assume that $\nu^{<\nu}>\nu$.
In particular, there exists a cardinal $\theta<\nu$ with $2^\theta\ge\nu$.
So, by \cite[Claim 3.1]{Sh:572} and \cite[Lemma 3.10]{Sh:572}, we may fix a $C$-sequence $\vec C=\langle C_\beta\mid \beta<\mu\rangle$ 
and a sequence of functions $\langle g_\beta:C_\beta\rightarrow\omega\mid \beta\in E^\mu_\nu\rangle$ 
such that:
\begin{itemize}
\item $\otp(C_\beta)=\cf(\beta)$ for all $\beta<\mu$;
\item  for every club $D\s\mu$,
there exists some $\beta\in E^\mu_\nu$ such that,
for every $n<\omega$,
$\sup\{ \delta\in\nacc(C_\beta)\cap D\mid g_\beta(\delta)=n\}=\beta$.
\end{itemize}
Fix a coloring $d:[\mu]^2\rightarrow\omega$ that satisfies $d(\alpha,\beta)=g_\beta(\min(C_\beta\setminus\alpha))$ for all $\beta\in E^\mu_\nu$ and $\alpha<\beta$.
Also, define a function $h:[\mu]^{<\nu}\rightarrow\nu$ via
$$h(z):=\sup\{\otp(C_{\beta}\cap\alpha)\mid (\alpha,\beta)\in[z]^2\}.$$

Next, define an ideal $J$ as follows:  $A$ is in $J$ iff $A\s\mu$ and there exists a club $D\s\mu$ such that,
for every $\beta\in D\cap A\cap E^\mu_\nu$, there exists $n<\omega$ such that $\sup\{ \delta\in\nacc(C_\beta)\cap D\mid g_\beta(\delta)=n\}<\beta$.
\begin{claim} $J$ is a $\mu$-additive normal ideal on $\mu$ with $E^\mu_\nu\notin J$.
\end{claim}
\begin{proof} 
By the choice of $\vec C$, $E^\mu_\nu\notin J$.
It is clear that for all $A\in J$ and $B\in[\mu]^{<\mu}$, $\mathcal P(A\cup B)\s J$.
Thus, it suffices to verify that $J$ is normal.
So, suppose that
$\langle A_i\mid i<\mu\rangle$ is a sequence of sets in $J$,
and we shall prove that $A:=\diagonaal_{i<\mu}A_i$ is in $J$.
For each $i<\mu$, fix a club $D_i$ witnessing that $A_i\in J$. 
We claim that $D:=\diagonal_{i<\mu}D_i$ witnesses that $A$ is in $J$.
Indeed, let $\beta\in D\cap A\cap E^\mu_\nu$ be arbitrary. 
Find $i<\beta$ such that $\beta\in A_i$. In particular, $\beta\in D_i\cap A_i\cap E^\mu_\nu$,
and hence there exists $n<\omega$ such that $\sup\{ \delta\in\nacc(C_\beta)\cap D\mid g_\beta(\delta)=n\}<\beta$.
\end{proof}

The rest of the proof now follows that of \cite[Theorem~2.3]{paper15}.
Given a sequence $\eta\in{}^{<\omega}\mu$, let
$$\mathcal D_\eta:=\{(i,j)\mid i<j<\ell(\eta)\ \&\ \eta(i)<\eta(j)\},$$
and whenever $\mathcal D_\eta\neq\emptyset$, set
\begin{itemize}
\item $\mathfrak m_\eta:=\max\{\otp(C_{\eta(j)}\cap\eta(i)) \mid (i,j)\in\mathcal D_\eta\}$;
\item $\mathcal P_\eta:=\{ (i,j)\in\mathcal D_\eta\mid \otp(C_{\eta(j)}\cap \eta(i))=\mathfrak m_\eta\}$;
\item $j_\eta:=\min\{ j\mid \exists i\ (i,j)\in\mathcal P_\eta\}$;
\item $\alpha_\eta:=\min\{\eta(i)\mid \exists j\ (i,j)\in\mathcal P_\eta\}$;
\item $\beta_\eta:=\eta(j_\eta)$.
\end{itemize}

Finally, define $d_0:{}^{<\omega}\mu\rightarrow\mu$ by letting for every $\eta\in{}^{<\omega}\mu$ with $\mathcal D_\eta\neq\emptyset$:
$$d_0(\eta):=\max\{0,j_\eta-d(\alpha_\eta,\beta_\eta)\}.$$

To verify this works,
suppose that we are given a sequence
$\langle (u_\alpha,v_\alpha,\sigma_\alpha)\mid\alpha\in A\rangle$ as in the statement of the lemma.
Note that, without loss of generality, we may assume
that $\alpha\notin\im(\sigma_\alpha)$ for all $\alpha\in A$.

For every $\alpha\in  A$, write
$a_\alpha:=\bigcup\{\im(\sigma)\mid \sigma\in u_\alpha\cup v_\alpha\}$, and
$x_\alpha:=a_\alpha\setminus\alpha$.
Let $\kappa$ be a large enough regular cardinal, and let
$\unlhd_\kappa$ be a well-ordering of $\mathcal H_\kappa$.
Let $\langle M_\delta\mid \delta<\mu\rangle$ be a continuous $\in$-chain of elementary submodels of $(\mathcal H_\kappa,\in,\unlhd_\kappa)$, each of size $\nu$,
such that $\nu\s M_0$ and $\{h,\langle a_\alpha\mid\alpha\in A\rangle\}\in M_0$.

Write $D:=\{ \delta<\mu\mid M_\delta\cap\mu=\delta\}$.
As $A\in J^+$, let us pick $\beta\in D\cap A\cap E^\mu_\nu$ 
such that $\sup\{ \delta\in\nacc(C_\beta)\cap D\mid g_\beta(\delta)=n\}=\beta$ for all $n<\omega$.
Put $\xi:=\sup(a_\beta\cap\beta)+1$. As $|a_\beta|<\cf(\beta)$, $\xi<\beta$.
Let $f:\nu\rightarrow\xi$ be the $\unlhd_\kappa$-least surjection.
From $|a_\beta|<\nu$ and regularity of $\nu$, let $i'<\nu$ be large enough such that $a_\beta\cap\beta\s f[i']$.
Write $n^*:=\ell(\sigma_\beta)$, $z:=f[i']$, $\epsilon:=h(a_\beta\cup z)$, and
$$A':=\{ \alpha\in A \mid a_\alpha\cap\alpha\s z, h(a_\alpha\cup z)=\epsilon\}.$$

Pick $\delta\in\nacc(C_\beta)\cap D$ above $\xi$ with $\otp(C_\beta\cap\delta)>\epsilon$
such that  $g_\beta(\delta)=n^*$.
As $\xi\in M_\delta$,  $A'\in M_\delta$. Since $\beta\in A'\setminus M_\delta$,  $\sup(A'\cap M_\delta)=\delta$.
So, let us pick $\alpha\in A'\cap M_\delta$ above $\max(C_\beta\cap\delta)$.

\begin{claim}\label{c222}
\begin{enumerate}
\item $h(a_\alpha\cup z)=\epsilon$;
\item $a_\alpha\cap\alpha\s z$;
\item $x_\alpha\s (\max(C_\beta\cap\delta),\delta)$. In particular, $\otp(C_\beta\cap\alpha)>\epsilon$;
\item $d(\alpha,\beta)=\ell(\sigma_\beta)$.
\end{enumerate}
\end{claim}
\begin{proof} By the same proof of \cite[Claim~2.3.1]{paper15}.
\end{proof}
To see that the pair $(\alpha,\beta)$ is as sought,
suppose that we are given $\varrho\in u_{\alpha}$ and $\sigma\in v_{\beta}$, and let us show that $d_0(\eta)=\ell(\varrho)$ for $\eta:=\varrho{}^\frown\sigma$.

As $\alpha\in\im(\varrho)$ and $\beta\in\im(\sigma)$, there exist $\hat i<\hat j<\ell(\eta)$ such that $\eta(\hat i)=\alpha$ and $\eta(\hat j)=\beta$.
So $(\hat i,\hat j)$ witnesses that $\mathcal D_\eta\neq\emptyset$, and then by Claim \ref{c222}(3), we have $\mathfrak m_\eta\ge \otp(C_{\beta}\cap \alpha)>\epsilon$.

\begin{claim}\label{c2} For every $(i,j)\in\mathcal P_\eta$:
\begin{enumerate}
\item $\{\eta(i),\eta(j)\}\nsubseteq(a_{\alpha}\cup z)$, and $\{\eta(i),\eta(j)\}\nsubseteq(a_{\beta}\cup z)$;
\item If $\eta(i)\in a_{\alpha}$, then $\eta(j)\notin a_{\beta}\cap\beta$;
\item $\varrho\sqsubseteq\eta\restriction j$;
\item $\eta(j)=\beta$;
\item $\eta(i)\in x_\alpha$.
\end{enumerate}
\end{claim}
\begin{proof} By the same proof of \cite[Claim~2.3.2]{paper15}.
\end{proof}

As $\beta\notin\im(\sigma_\beta)$, we get from the minimality of $j_\eta$
together with Clauses (3) and (4) of the preceding claim that $$\eta\restriction(j_\eta+1)=\varrho{}^\frown\sigma_\beta{}^\frown\langle\beta\rangle.$$
So $\beta_\eta=\beta$ and $j_\eta=\ell(\varrho{}^\frown\sigma_\beta)$. By Clause (5) of the preceding claim, $\alpha_\eta\in x_\alpha$.
Then, by  Claim \ref{c222}(3), we get that $\otp(C_\beta\cap \alpha_\eta)=\otp(C_\beta\cap\alpha)$.
Recalling that $\min(x_\alpha)=\alpha\in\im(\varrho)$, Claims  \ref{c2}(5) and \ref{c222}(3) then imply that $\alpha_\eta=\alpha$.

Recalling Claim \ref{c222}(4), we altogether infer that
$$\eta\restriction (j_\eta-d(\alpha_\eta,\beta_\eta))=(\varrho{}^\frown\sigma)\restriction(\ell(\varrho{}^\frown\sigma_\beta)-\ell(\sigma_\beta))=\varrho.$$
So, $d_0(\sigma)=\max\{0,j_\eta-d(\alpha_\eta,\beta_\eta)\}=\ell(\varrho)$, as sought.
\end{proof}

\section{Clause~(2) of Theorem~C}\label{mainthm5}
In this section, we suppose that $\chi\in\reg(\kappa)$ is a cardinal satisfying $\chi^{+}<\kappa$,
and there exists a stationary subset of $E^\kappa_{\ge\chi}$ that does not reflect.
We shall construct a witness to $\pl_1(\kappa,\kappa,\chi)$.
The proof is split into two cases: $\chi^{++}<\kappa$ and $\chi^{++}=\kappa$.

\subsection{Case I}\label{ssectioncase1} In this subsection, we suppose that $\chi^{++}<\kappa$.
Note that, by Proposition~\ref{prop219}(2), the result of this subsection cannot be improved.

\begin{lemma} There exists $\nu\in\reg(\kappa)\setminus\chi$  with $\nu^+<\kappa$ 
and a stationary subset $\Gamma\s E^\kappa_{\ge\chi}\cap E^\kappa_{\neq\nu^+}$ that does not reflect.
\end{lemma}
\begin{proof} By the hypothesis of this section, let us fix a stationary subset $R\s E^\kappa_{\ge\chi}$ that does not reflect.

If $R\cap\reg(\kappa)$ is stationary, then we may simply let $\nu:=\chi$ and $\Gamma:=R\cap\reg(\kappa)\setminus(\nu^++1)$.
Next, suppose that $R\cap\reg(\kappa)$ is nonstationary, and use Fodor's lemma to fix a regular cardinal $\theta\ge\chi$ 
for which $R\cap E^\kappa_\theta$ is stationary.

$\br$ If $\theta^+<\kappa$, then we let $\nu:=\theta$. It follows that $\nu^+<\kappa$, 
and $E^\kappa_\theta\cap E^\kappa_{\nu^+}=\emptyset$,
so that $\Gamma:=R\setminus E^\kappa_{\nu^+}$ is as sought.

$\br$ If $\theta^+=\kappa$, then we let $\nu:=\chi$.
As $\chi^{++}<\kappa$, we infer that $\nu^+<\theta<\kappa$, so that $E^\kappa_\theta\cap E^\kappa_{\nu^+}=\emptyset$
and $\Gamma:=R\setminus E^\kappa_{\nu^+}$ is as sought.
\end{proof}

Let $\nu$ and $\Gamma$ be given by the preceding lemma. 
Set $\mu:=\nu^+$, so that $\Gamma\cap E^\kappa_\mu=\emptyset$.
Fix a surjection $g:\kappa\rightarrow\kappa\times\kappa$
such that $G_{\eta,\tau}:=\{\delta\in \Gamma\mid g(\delta)=(\eta,\tau)\}$ is stationary for every $(\eta,\tau)\in\kappa\times\kappa$.
Fix another surjection $h:\kappa\rightarrow\mu$
such that  $H_i:=\{\alpha\in \Gamma\mid h(\alpha)=i\}$ is stationary for every $i<\mu$.

As $\Gamma$ is nonreflecting,
let $\vec C=\langle C_\alpha\mid\alpha<\kappa\rangle$ be a sequence such that $C_{\alpha+1}=\{\alpha\}$ for every $\alpha<\kappa$,
and such that, for every $\alpha\in\acc(\kappa)$, $C_\alpha$ is a club in $\alpha$ with $\acc(C_\alpha)\cap\Gamma=\emptyset$.
By a club-guessing theorem due to Shelah (cf.~\cite[Remark~1.5 and Lemma~2.5]{paper29}), we may also assume that, for every club $D\s\kappa$, there exists $\gamma\in\Gamma$ with $\sup(\nacc(C_\gamma)\cap D)=\gamma$.
Recalling Subsection~\ref{subsectionwalks},
we now let $\Tr,\tr,\lambda$ and $\rho_2$ be the characteristic functions of walking along $\vec C$,
and let $\eta_{\alpha,\beta}$ be the notation established in Definition~\ref{etanotation}.
In addition, we consider yet another function $\tr_h:[\kappa]^2\rightarrow{}^{<\omega}\mu$ which is defined via $\tr_h(\alpha,\beta):=h\circ\tr(\alpha,\beta)$.

Appeal to Lemma~\ref{lemma42} to fix a map $d_0:{}^{<\omega}\mu\rightarrow\omega$ and its corresponding $\mu$-additive proper ideal $J$.
Define $c:[\kappa]^2\rightarrow\kappa\times\kappa$ via 
$$c(\alpha,\beta):=g(\Tr(\alpha,\beta)(d_0(\tr_h(\alpha,\beta)))).$$

We are finally ready to define our transformation.
\begin{defn}
Define $\mathbf t:[\kappa]^2\rightarrow[\kappa]^3$
by letting, for all $(\alpha,\beta)\in[\kappa]^2$, $\mathbf t(\alpha,\beta):=(\tau,\alpha^*,\beta^*)$ provided that the following conditions are met:
\begin{itemize}
\item $(\eta,\tau):=c(\alpha,\beta)$ and $\max\{\eta+1,\tau\}<\alpha$,
\item $\beta^*=\Tr(\alpha,\beta)(\eta_{\alpha,\beta})$ is $>\alpha$, and
\item $\alpha^*=\Tr(\eta+1,\alpha)(\eta_{\eta+1,\alpha})$.
\end{itemize}
Otherwise, let $\mathbf t(\alpha,\beta):=(0,\alpha,\beta)$.
\end{defn}

To verify that $\mathbf t$ witnesses $\pl_1(\kappa,\kappa,\chi)$,
suppose that we are given a family $\mathcal A\s[\kappa]^{<\chi}$ consisting of $\kappa$ many pairwise disjoint sets.
Fix a sequence  $\vec x=\langle x_\delta\mid \delta<\kappa\rangle$ such that,
for all $\delta<\kappa$, $x_\delta\in\mathcal A$ with $\min(x_\delta)>\delta$.

\begin{defn}\label{defseta} For $\eta<\kappa$, $S_{\eta}$ denotes the set of all $\epsilon<\kappa$ 
with the property that, for every $\varsigma<\kappa$, there exist $I\in J^+$ 
and a sequence $\langle \beta_i\mid i\in I\rangle\in\prod_{i\in I}H_i\setminus\varsigma$, such that,
for all $i\in I$ and $\beta\in x_{\beta_i}$:
\begin{enumerate}
\item[(i)]  $i\in\im(\tr_h(\epsilon,\beta))$;
\item[(ii)] $\lambda(\epsilon,\beta)=\eta$;
\item[(iii)]  $\rho_2(\epsilon,\beta)=\eta_{\epsilon,\beta}$.
\end{enumerate}
\end{defn}

\begin{lemma}  There exists $\eta<\kappa$ for which $S_\eta$ is stationary.
\end{lemma}
\begin{proof} By the pressing down lemma, it suffices to prove that, for every club $D\s\kappa$,
there exist $\epsilon\in D$ and $\eta<\epsilon$ for which $\epsilon\in S_\eta$.
Thus, let $D$ be an arbitrary club in $\kappa$.

Define a function $f:\Gamma\rightarrow\kappa$ via
$$f(\delta):=\sup\{\lambda(\delta,\beta)\mid \beta\in x_\delta\}.$$

By Fact~\ref{fact1} and since $|x_\delta|<\chi\le\cf(\delta)$ for all $\delta\in \Gamma$,
$f$ is regressive. So, for all $i<\mu$, let us pick a stationary subset $\bar H_i\subseteq H_i$ such that $f\restriction \bar H_i$ is constant.
Set $\zeta:=\sup(f[\bigcup_{i<\mu}\bar H_i])$.
Now, by the club-guessing feature of $\vec C$, let us pick $\gamma\in\Gamma$ with $\sup(\nacc(C_\gamma)\cap (D\setminus\zeta))=\gamma$.

Let $\varsigma<\kappa$. Fix a sequence $\langle \beta^\varsigma_i\mid i<\mu\rangle\in\prod_{i<\mu}\bar H_i\setminus\max\{\gamma+1,\varsigma\}$.
For every $i<\mu$, by Fact~\ref{fact1}, $\lambda(\gamma,\beta_i^\varsigma)<\gamma$,
so as $\gamma\in\Gamma\s E^\kappa_{\neq\mu}$ and as $J$ is a $\mu$-additive proper ideal on $\mu$,
we may fix $I^\varsigma\in J^+$ along with some ordinal $\xi^\varsigma<\gamma$ such that $\lambda(\gamma,\beta_i^\varsigma)\le\xi^\varsigma$ for all $i\in I^\varsigma$.
Then, pick a large enough $\epsilon^\varsigma\in\nacc(C_\gamma)\cap D$ such that $\sup(C_\gamma\cap\epsilon^\varsigma)>\max\{\xi^\varsigma,\zeta\}$.

Next, by the pigeonhole principle, let us fix $\epsilon\in\nacc(C_\gamma)\cap D$ for which $\Sigma:=\{\varsigma<\kappa\mid \epsilon^\varsigma=\epsilon\}$
is cofinal in $\kappa$.
Put $\eta:=\sup(C_\gamma\cap\epsilon)$, so that $\eta<\epsilon$.

We already know that $\epsilon\in D$. To see that $\epsilon\in S_\eta$, let $\varsigma<\kappa$ be arbitrary. 
By increasing $\varsigma$, we may assume that $\varsigma\in\Sigma$.
Let $i\in I^\varsigma$ and $\beta\in x_{\beta^\varsigma_i}$ be arbitrary. 
As $\beta_i^\varsigma\in H_i$, it suffices to show that:
\begin{enumerate}
\item[(i')] $\tr(\epsilon,\beta)=\tr(\beta_i^\varsigma,\beta){}^\smallfrown\tr(\epsilon,\beta_i^\varsigma)$;
\item[(ii')] $\lambda(\epsilon,\beta)=\eta$;
\item[(iii')]  $\rho_2(\epsilon,\beta)=\eta_{\epsilon,\beta}$.
\end{enumerate}

We have:
$$\lambda(\beta_i^\varsigma,\beta)\le f(\beta_i^\varsigma)\le\zeta\le\max\{\lambda(\gamma,\beta_i^\varsigma),\zeta\}\le\max\{\xi^\varsigma,\zeta\}<\eta<\epsilon<\gamma<\beta_i^\varsigma<\beta.$$
It thus follows from Fact~\ref{fact2} that Clause~(i') is satisfied. It also follows from Fact~\ref{fact2}  that 
$\tr(\epsilon,\beta_i^\varsigma)=\tr(\gamma,\beta_i^\varsigma){}^\smallfrown\tr(\epsilon,\gamma)$,
so that altogether 
$$\tr(\epsilon,\beta)=\tr(\beta_i^\varsigma,\beta){}^\smallfrown\tr(\gamma,\beta_i^\varsigma){}^\smallfrown\tr(\epsilon,\gamma).$$
By Lemma~\ref{fact3} and the above equation,
$$\lambda(\epsilon,\beta)=\max\{\lambda(\beta_i^\varsigma,\beta),\lambda(\gamma,\beta_i^\varsigma),\lambda(\epsilon,\gamma)\}.$$

Recall that $\max\{\lambda(\beta_i^\varsigma,\beta),\lambda(\gamma,\beta_i^\varsigma)\}\le\max\{\zeta,\xi^\varsigma\}<\eta$.
As $\epsilon\in C_\gamma$, we infer that $\lambda(\epsilon,\gamma)=\sup(C_\gamma\cap\epsilon)=\eta$.
In effect, $\lambda(\epsilon,\beta)=\eta$ and $\rho_2(\epsilon,\beta)=\eta_{\epsilon,\beta}$.
\end{proof}

Let $\eta$ be given by the preceding lemma.
Let $D$ be a club in $\kappa$ such that, for all $\delta\in D$, there exists $M_\delta\prec\mathcal H_{\kappa^{+}}$ containing 
the parameter $p:=\{\Gamma,S_\eta,\vec x,\vec C,h,\mu\}$
and satisfying $M_\delta\cap\kappa=\delta$.
Finally, let $$S^*:=S_{\eta}\cap\diagonal_{\tau<\kappa}\acc^+\left(G_{\eta,\tau}\cap\bigcap\nolimits_{j<\mu}\acc^+(H_j\cap D)\right).$$
\begin{lemma}\label{wrappingup}
Let $(\tau^*,\alpha^*,\beta^*)\in \kappa\circledast S^*\circledast S^*$.
There exists $(a,b)\in[\mathcal A]^2$ such that $\mathbf t[a\times b]=\{(\tau^*,\alpha^*,\beta^*)\}$.
\end{lemma}
\begin{proof} As $\beta^*\in S^*\s S_{\eta}$, let us pick $I\in J^+$ and a sequence $\langle \beta_i\mid i\in I\rangle\in\prod_{i\in I}H_i\setminus(\beta^*+1)$ such that, for all $i\in I$ and $\beta\in x_{\beta_i}$:
\begin{enumerate}
\item $i\in\im(\tr_h(\beta^*,\beta))$;
\item $\lambda(\beta^*,\beta)=\eta$;
\item  $\rho_2(\beta^*,\beta)=\eta_{\beta^*,\beta}$.
\end{enumerate}

As $(\tau^*,\beta^*)\in\kappa\circledast S^*$,
pick a large enough $\varepsilon\in \left(G_{\eta,\tau^*}\cap\bigcap_{j<\mu}\acc^+(H_j\cap D)\right)\cap \beta^*$ such that $\sup(C_{\beta^*}\cap\varepsilon)>\alpha^*$. In particular, $\lambda(\varepsilon,\beta^*)>\alpha^*>\eta$.

For all $j<\mu$, as $\varepsilon\in\Gamma\cap\acc^+(H_j\cap D)$, Fact~\ref{fact1} entails that we may pick a large enough $\delta_j\in H_j\cap D\cap\varepsilon$ such that $\delta_j>\lambda(\varepsilon,\beta^*)$.
As $M_{\delta_j}$ contains $p$, we have that $S_{\eta}\in M_{\delta_j}$.
As $\delta_j\in\Gamma$, Fact~\ref{fact1} entails that $\varsigma_j:=\max\{\alpha^*,\lambda(\varepsilon,\beta^*),\lambda(\delta_j,\varepsilon)\}+1$ is smaller than $\delta_j$.
Since $\alpha^*\in M_{\delta_j}\cap S_{\eta}$, we may then find $\alpha_j\in M_{\delta_j}\cap(\bigcup_{i<\mu}H_i)\setminus\varsigma_j$ such that, for all $\alpha\in x_{\alpha_j}$:
\begin{enumerate}
\item[(2')] $\lambda(\alpha^*,\alpha)=\eta$;
\item[(3')] $\rho_2(\alpha^*,\alpha)=\eta_{\alpha^*,\alpha}$.
\end{enumerate}

Note that from $\alpha_j\in M_{\delta_j}$, it follows that $\sup(x_{\alpha_j})<\delta_j$. 
Write  $a_j:=x_{\alpha_j}$ and $b_i:=x_{\beta_i}$.
Let $(i,j,\alpha,\beta)\in I\times\mu\times a_j\times b_i$ be arbitrary. Then:
$$\eta<\eta+1<\alpha^*<\varsigma_j\le\alpha_j<\alpha<\delta_j<\varepsilon<\beta^*<\beta_i<\beta.$$
In particular, Fact~\ref{fact2} yields the following conclusions:
\begin{itemize}
\item[(a)] from $\lambda(\beta^*,\beta)=\eta<\alpha<\beta^*<\beta$, we have $\tr(\alpha,\beta)=\tr(\beta^*,\beta){}^\smallfrown\tr(\alpha,\beta^*)$;
\item[(b)] from $\lambda(\varepsilon,\beta^*)<\varsigma_j<\alpha<\varepsilon<\beta^*$, we have $\tr(\alpha,\beta^*)=\tr(\varepsilon,\beta^*){}^\smallfrown\tr(\alpha,\varepsilon)$;
\item[(c)] from $\lambda(\delta_j,\varepsilon)<\varsigma_j<\alpha<\delta_j<\varepsilon$, we have $\tr(\alpha,\varepsilon)=\tr(\delta_j,\varepsilon){}^\smallfrown \tr(\alpha,\delta_j)$.
\end{itemize}
So that, altogether,
$$\tr(\alpha,\beta)=\tr(\beta^*,\beta){}^\smallfrown\tr(\varepsilon,\beta^*){}^\smallfrown\tr(\delta_j,\varepsilon){}^\smallfrown \tr(\alpha,\delta_j).$$

In addition, from $\lambda(\alpha^*,\alpha)=\eta<\eta+1<\alpha^*<\alpha$, we infer that 
\begin{itemize}
\item[(d)] $\tr(\eta+1,\alpha)=\tr(\alpha^*,\alpha){}^\smallfrown\tr(\eta+1,\alpha^*)$.
\end{itemize}

For each $i\in I$, denote $u_i:=\{ \tr_h(\varepsilon,\beta)\mid \beta\in b_i\}$.
For each $j<\mu$, denote $v_j:=\{ \tr_h(\alpha,\varepsilon)\mid\alpha\in a_j\}$.
\begin{claim} 
\begin{enumerate}
\item[(i)] For every $i\in I$, $i\in\im(\varrho)$ for all $\varrho\in u_i$;
\item[(ii)] For every $j<\mu$, there exists $\sigma_j\in{}^{<\omega}\mu$ such that $\sigma_j{}^\smallfrown\langle j\rangle\sqsubseteq \sigma$ for all $\sigma\in v_j$.
\end{enumerate}
\end{claim}
\begin{proof} (i) For all $\beta\in b_i$, $\tr_h(\varepsilon,\beta)=\tr_h(\beta^*,\beta){}^\smallfrown\tr_h(\varepsilon,\beta^*)$,
so the conclusion follows from Clause~(1).

(ii) Since $\delta_j\in H_j$,
by Clause~(c) above, $\tr_h(\delta_j,\varepsilon){}^\smallfrown\langle j\rangle\sqsubseteq\sigma$ for all $\sigma\in v_j$.
\end{proof}

Next, by the choice of $d_0$, fix $(i,j)\in[I]^2$ 
such that $d_0(\varrho{}^\smallfrown\sigma)=\ell(\varrho)$ for all $\varrho\in u_i$ and $\sigma\in v_j$.
Set $a:=a_j$ and $b:=b_i$, so that $(a,b)\in[\mathcal A]^2$.

To see that $\mathbf t[a\times b]=\{(\tau^*,\alpha^*,\beta^*)\}$, fix arbitrary $\alpha\in a$ and $\beta\in b$.
\begin{claim} $c(\alpha,\beta)=(\eta,\tau^*)$.
\end{claim}
\begin{proof} 
Write $\varrho:=\tr_h(\varepsilon,\beta)$ and $\sigma:=\tr_h(\alpha,\varepsilon)$.
Then:
\begin{itemize}
\item $\tr_h(\alpha,\beta)=\varrho{}^\smallfrown\sigma$;
    \item $d_0(\tr_h(\alpha,\beta))=\ell(\varrho)=\ell(\tr(\varepsilon,\beta))=\rho_2(\varepsilon,\beta)$;
    \item $\Tr(\alpha,\beta)(d_0(\tr_h(\alpha,\beta)))=\Tr(\alpha,\beta)(\rho_2(\varepsilon,\beta))=\varepsilon$.
\end{itemize}
So, $c(\alpha,\beta)=g(\Tr(\alpha,\beta)(d_0(\tr_h(\alpha,\beta))))=g(\varepsilon)=(\eta,\tau^*)$.
\end{proof}

By Clause~(a) above, $\tr(\alpha,\beta)=\tr(\beta^*,\beta){}^\smallfrown\tr(\alpha,\beta^*)$,
so Clause~(3) above implies that $\eta_{\alpha,\beta}=\eta_{\beta^*,\beta}=\rho_2(\beta^*,\beta)$.

By Clause~(d) above, $\tr(\eta+1,\alpha)=\tr(\alpha^*,\alpha){}^\smallfrown\tr(\eta+1,\alpha^*)$,
so Clause~(3') above implies that $\eta_{\eta+1,\alpha}=\eta_{\alpha^*,\alpha}=\rho_2(\alpha^*,\alpha)$.
Altogether, $\mathfrak t(\alpha,\beta)=(\tau^*,\alpha^*,\beta^*)$.
\end{proof}

\subsection{Case II}\label{subsectioncase2} In this subsection, we suppose that $\chi^{++}=\kappa$.
Denote $\mu:=\chi^+$. It is clear that $\pl_1(\kappa,\kappa,\chi)$ is equivalent to $\pl_1(\kappa,\mu,\chi)$,
so we shall focus on constructing a witness to the latter.
Denote $\Gamma:=E^\kappa_\mu$.

Fix a function $h:\kappa\rightarrow\mu$ such that, for every $i<\mu$, $H_i:=\{\alpha\in\Gamma\mid h(\alpha)=i\}$ is stationary.
By a club-guessing theorem due to Shelah \cite[\S2]{Sh:365} (cf.~\cite{soukup2010club}), we may fix a $C$-sequence $\vec C=\langle C_\alpha\mid\alpha<\kappa\rangle$
such that:
\begin{itemize}
\item for every $\alpha<\kappa$, $\otp(C_\alpha)=\cf(\alpha)$;
\item for every club $D\s\kappa$ and every $i<\mu$, there exists $\gamma\in H_i$ with $\sup(\nacc(C_\gamma)\cap D)=\gamma$.
\end{itemize}

Note that $\acc(C_\alpha)\cap\Gamma=\emptyset$ for all $\alpha<\kappa$.
Recalling Subsection~\ref{subsectionwalks},
we now let $\Tr,\tr,\lambda$ and $\rho_2$ be the characteristic functions of walking along $\vec C$.
In addition, we consider yet another function $\tr_h:[\kappa]^2\rightarrow{}^{<\omega}\mu$ which is defined via $\tr_h(\alpha,\beta):=h\circ\tr(\alpha,\beta)$.

Fix a sequence $\langle Z_\epsilon\mid \epsilon<\kappa\rangle$ of elements of $[\mu]^\mu$ such that,
for every $(\alpha,\beta)\in[\mu]^\mu$, $|Z_\alpha\cap Z_\beta|<\mu$.

\begin{defn}\label{xialphabeta}
For every ordinal $\xi<\mu$ and a pair $(\alpha,\beta)\in[\kappa]^2$, let
$$\xi^{\alpha,\beta}:=\min\{ n<\omega\mid \xi\in Z_{\Tr(\alpha,\beta)(n)}\text{ or }n=\rho_2(\alpha,\beta)+1\}.$$
\end{defn}

\begin{lemma} There exists a map $d_1:{}^{<\omega}\mu\rightarrow\omega\times\mu\times\mu\times\mu$,
such that, for every $(\tau,\xi,\phi)\in\mu\times\mu\times\mu$ and every sequence $\langle (u_i,v_i,\sigma_i)\mid i<\mu\rangle$, with
\begin{enumerate}
\item  $u_i$ and $v_i$ are nonempty elements of $[{}^{<\omega}\mu]^{<\chi}$;
\item $i\in\im(\varrho)$ for all $\varrho\in u_i$;
\item $\sigma_j{}^\smallfrown\langle j\rangle\sqsubseteq \sigma$ for all $\sigma\in v_j$,
\end{enumerate}
there exist $(i,j)\in[\mu]^2$
satisfying that $d_1(\varrho{}^\smallfrown\sigma)=(\ell(\varrho),\tau,\xi,\phi)$ for all $\varrho\in u_i$ and $\sigma\in v_j$.
\end{lemma}
\begin{proof} Let $d:{}^{<\omega}\mu\rightarrow\omega\times\mu\times\mu\times\mu$ be given by Fact~\ref{pl6} using $\nu:=\chi$.
Fix a bijection $\pi:\mu\leftrightarrow\mu\times\mu\times\mu$.
Then, define $d_1:{}^{<\omega}\mu\rightarrow\omega\times\mu\times\mu\times\mu$ by letting 
$d_1(\sigma):=(n,\tau,\xi,\phi)$ whenever $d(\sigma)=(n,i,j,\gamma)$ and $\pi(\gamma)=(\tau,\xi,\phi)$.
Evidently, $d_1$ is as sought.
\end{proof}
Let $d_1:{}^{<\omega}\mu\rightarrow\omega\times\mu\times\mu\times\mu$ be given by the preceding lemma.
For every nonzero $\epsilon<\kappa$, fix a surjection $\psi_\epsilon:\mu\rightarrow\epsilon$.
We are now ready to define our transformation.
\begin{defn}
Define $\mathbf t:[\kappa]^2\rightarrow[\kappa]^3$
by letting, for all $(\alpha,\beta)\in[\kappa]^2$, $\mathbf t(\alpha,\beta):=(\tau^*,\alpha^*,\beta^*)$ provided that,
for $(n,\tau,\xi,\phi):=d_1(\tr_h(\alpha,\beta))$, all of the following conditions are met:
\begin{itemize}
\item $\beta^*=\Tr(\alpha,\beta)(n)$ is $>\alpha$,
\item $\eta:=\psi_{\beta^*}(\phi)$ satisfies that $\eta+1<\alpha$,
\item $\alpha^*=\Tr(\eta+1,\alpha)(\xi^{\eta+1,\alpha})$, and
\item $\tau^*=\tau<\alpha^*$.
\end{itemize}
Otherwise, let $\mathbf t(\alpha,\beta):=(0,\alpha,\beta)$.
\end{defn}

To verify that $\mathbf t$ witnesses $\pl_1(\kappa,\mu,\chi)$,
suppose that we are given a family $\mathcal A\s[\kappa]^{<\chi}$  consisting of $\kappa$ many pairwise disjoint sets.
\begin{lemma} For every $i<\mu$, there exist an ordinal $\zeta_i<\kappa$ and a sequence $\vec{x^i}=\langle x^i_\gamma\mid \gamma\in \Gamma_i\rangle$ such that:
\begin{itemize}
\item $\Gamma_i$ is a stationary subset of $\kappa$;
\item for all $\gamma\in\Gamma_i$, $x_\gamma^i\in\mathcal A$ with $\min(x^i_\gamma)>\gamma$;
\item for all $\gamma\in\Gamma_i$ and $\beta\in x^i_\gamma$, $\lambda(\gamma,\beta)=\zeta_i$ and $i\in\im(\tr_h(\gamma,\beta))$.
\end{itemize}
\end{lemma}
\begin{proof} Let $i<\mu$. By the pressing down lemma, it suffices to prove that, for every club $D\s\kappa$,
there exist $\gamma\in D$, $\zeta<\gamma$ and $x\in\mathcal A$ with $\min(x)>\gamma$ such that 
$\lambda(\gamma,\beta)=\zeta$ and $i\in\tr_h(\gamma,\beta)$ for all $\beta\in x$.
Thus, let $D$ be an arbitrary club in $\kappa$.

By the choice of $\vec C$, fix $\delta\in H_i$ such that $\sup(\nacc(C_\delta)\cap D)=\delta$. 
Then, fix any $x\in\mathcal A$ with $\min(x)>\delta$. As $\delta\in\Gamma$ and $|x|<\chi<\cf(\delta)$,
Fact~\ref{fact1} entails that we may find a large enough $\gamma\in \nacc(C_\delta)\cap D$ with $\zeta:=\sup(C_\delta\cap\gamma)$
being greater than $\sup_{\beta\in x}\lambda(\delta,\beta)$.
Now, for every $\beta\in x$, we have $\lambda(\delta,\beta)<\zeta<\gamma<\delta<\beta$,
so, by Fact~\ref{fact2}, $\tr(\gamma,\beta)=\tr(\delta,\beta){}^\smallfrown\tr(\gamma,\delta)$.
In particular, $i=h(\delta)\in\im(\tr_h(\gamma,\beta))$
Next, by Lemma~\ref{fact3}, $\lambda(\gamma,\beta)=\max\{\lambda(\delta,\beta),\lambda(\gamma,\delta)\}$.
As $\gamma\in C_\delta$, we have $\lambda(\gamma,\delta)=\sup(C_\delta\cap\gamma)>\zeta=\lambda(\delta,\beta)$, so that, altogether,
$\lambda(\gamma,\beta)=\zeta$.
\end{proof}

For each $i<\mu$, let $\zeta_i$ and $\vec{x^i}=\langle x^i_\gamma\mid \gamma\in \Gamma_i\rangle$ be given by the preceding lemma. 
For notational simplicity, we shall drop the superscript $i$, writing $\vec{x^i}=\langle x_\gamma\mid \gamma\in \Gamma_i\rangle$.\footnote{This is formally legitimate
provided that the stationary sets in $\langle \Gamma_i\mid i<\mu\rangle$ are pairwise disjoint. Now, as $\mu$ is regular,
for any sequence $\langle \Gamma_i\mid i<\mu\rangle$ of stationary subsets of $\mu^+$,
there exists a sequence $\langle \bar\Gamma_i\mid i<\mu\rangle$ of pairwise disjoint stationary sets such that $\bar\Gamma_i\s\Gamma_i$ for all $i<\mu$ (cf.~\cite{paper47}). 
So, we may as well assume that the original sequence consists of pairwise disjoint sets.}
Set $\zeta:=\sup_{i<\mu}\zeta_i$.

\begin{defn} 
For $\eta<\kappa$ and $\xi,\phi<\mu$, $S_{\eta,\xi,\phi}$ denotes the set of all $\epsilon\in\Gamma$ 
with the property that, for every $\varsigma<\kappa$, there exists
a sequence $\langle \beta_i\mid i<\mu\rangle\in\prod_{i<\mu}\Gamma_i\setminus\varsigma$ such that,
for all $i<\mu$ and $\beta\in x_{\beta_i}$:
\begin{enumerate}
\item[(i)]  $\tr(\epsilon,\beta)=\tr(\beta_i,\beta){}^\smallfrown\tr(\epsilon,\beta_i)$;
\item[(ii)] $\lambda(\epsilon,\beta)<\epsilon$;
\item[(iii)]  if $i=0$, then $\lambda(\epsilon,\beta)=\eta=\psi_\epsilon(\phi)$, and $\rho_2(\epsilon,\beta)=\xi^{\epsilon,\beta}$.
\end{enumerate}
\end{defn}
\begin{lemma} There exist $\eta<\kappa$ and $\xi,\phi<\mu$ for which $S_{\eta,\xi,\phi}$ is stationary.
\end{lemma}
\begin{proof} For all $i<\mu$ and $\varsigma<\kappa$, denote $\beta_i^\varsigma:=\min(\Gamma_i\setminus\varsigma)$.

Let $\epsilon\in\Gamma\setminus(\zeta+1)$.
For every $\varsigma$ in the interval $(\epsilon,\kappa)$, define $f_\epsilon^\varsigma:\mu\rightarrow\epsilon$ via $f_\epsilon^\varsigma(i):=\max\{\zeta_i,\lambda(\epsilon,\beta_i^{\varsigma})\}$.
Now, find $\eta_\epsilon<\epsilon$ and $\phi_\epsilon,\xi_\epsilon<\mu$ for which 
$$\Sigma_\epsilon:=\left\{\varsigma\in(\epsilon,\kappa)\mid f^\varsigma_\epsilon(0)=\eta_\epsilon=\psi_\epsilon(\phi_\epsilon)\ \&\ \xi_\epsilon\in Z_\epsilon\setminus\bigcup\{Z_\tau\mid \tau\in\im(\tr(\epsilon,\beta)), \beta\in x_{\beta^\varsigma_{0}} \}\right\}$$
 is cofinal in $\kappa$.

Finally, find $\eta,\xi,\phi$ for which $S:=\{\epsilon\in\Gamma\setminus(\zeta+1)\mid (\eta,\xi,\phi)=(\eta_\epsilon,\xi_\epsilon,\phi_\epsilon)\}$ is stationary.
We claim that $S\s S_{\eta,\xi,\phi}$. Let $\epsilon\in S$ be arbitrary;
to see that $\epsilon\in S_{\eta,\xi,\phi}$, let $\varsigma<\kappa$ be arbitrary. 
By increasing $\varsigma$, we may assume that $\varsigma\in\Sigma_\epsilon$.
Let $i<\mu$ and $\beta\in x_{\beta^\varsigma_i}$ be arbitrary.
We will show that:

\begin{enumerate}
\item[(i')] $\tr(\epsilon,\beta)=\tr(\beta_i^\varsigma,\beta){}^\smallfrown\tr(\epsilon,\beta_i^\varsigma)$;
\item[(ii')] $\lambda(\epsilon,\beta)=f_\epsilon^\varsigma(i)$;
\item[(iii')] if $i=0$, then $\rho_2(\epsilon,\beta)=\xi^{\epsilon,\beta}$.
\end{enumerate}

As $\lambda(\beta_i^\varsigma,\beta)=\zeta_i<\epsilon<\beta_i^\varsigma<\beta$,
it  follows from Fact~\ref{fact2} that Clause~(i') is satisfied,
and it follows from Lemma~\ref{fact3} that 
$$\lambda(\epsilon,\beta)=\max\{\lambda(\beta_i^\varsigma,\beta),\lambda(\epsilon,\beta_i^\varsigma)\}=\max\{\zeta_i,\lambda(\epsilon,\beta_i^\varsigma)\}=f_\epsilon^\varsigma(i).$$
In addition, from $\varsigma\in\Sigma_\epsilon$, Clause~(iii') is satisfied.
\end{proof}

Let $\eta,\xi,\phi$ be given by the preceding lemma.
Let $D$ be a club in $\kappa$ such that, for all $\delta\in D$, there exists $M_\delta\prec\mathcal H_{\kappa^{+}}$ containing 
the parameter $p:=\{\Gamma,S_{\eta,\xi,\phi},\vec{x^0},\vec C,h,\mu\}$
and satisfying $M_\delta\cap\kappa=\delta$.
Finally, let $$S^*:=S_{\eta,\xi,\phi}\cap\bigcap\nolimits_{j<\mu}\acc^+(H_j\cap D).$$
\begin{lemma}
Let $(\tau^*,\alpha^*,\beta^*)\in\mu\circledast S^*\circledast S^*$. There exists $(a,b)\in[\mathcal A]^2$ such that $\mathbf t[a\times b]=\{(\tau^*,\alpha^*,\beta^*)\}$.
\end{lemma}
\begin{proof} As $\beta^*\in S^*\s S_{\eta,\xi,\phi}$, let us fix a sequence $\langle \beta_i\mid i<\mu\rangle\in\prod_{i<\mu}\Gamma_i\setminus(\beta^*+1)$ 
such that, for all $i<\mu$ and $\beta\in x_{\beta_i}$:
\begin{enumerate}
\item $\tr(\beta^*,\beta)=\tr(\beta_i,\beta){}^\smallfrown\tr(\beta^*,\beta_i)$;
\item $\lambda(\beta^*,\beta)<\beta^*$;
\item $\psi_{\beta^*}(\phi)=\eta$.
\end{enumerate}

For each $i<\mu$, $|x_{\beta_i}|<\chi<\cf(\beta^*)$, so we may define a function $f:\mu\rightarrow\beta^*$ via $f(i):=\sup\{\lambda(\beta^*,\beta)\mid \beta\in x_{\beta_i}\}$.
For all $j<\mu$, as $\beta^*\in\Gamma\cap\acc^+(H_j\cap D)$, we may pick a large enough $\delta_j\in H_j\cap D\cap\beta^*$ such that $\delta_j>\max\{\alpha^*,\sup_{i<j}f(i)\}$.
As $M_{\delta_j}$ contains $p$, we have that $S_{\eta,\xi,\phi}\in M_{\delta_j}$.
As $\delta_j\in\Gamma$, Fact~\ref{fact1} entails that $\varsigma_j:=\max\{\alpha^*,\sup_{i<j}f(i),\lambda(\delta_j,\beta^*)\}+1$ is smaller than $\delta_j$.
Since $\alpha^*\in M_{\delta_j}\cap S_{\eta,\xi,\phi}$, we may then find $\alpha_j\in M_{\delta_j}\cap \Gamma_0\setminus\varsigma_j$ such that, for all $\alpha\in x_{\alpha_j}$:
\begin{enumerate}
\item[(4)] $\lambda(\alpha^*,\alpha)=\eta$ and $\rho_2(\alpha^*,\alpha)=\xi^{\alpha^*,\alpha}$.
\end{enumerate}

Note that from $\alpha_j\in M_{\delta_j}$, it follows that $\sup(x_{\alpha_j})<\delta_j$. 
Write  $a_j:=x_{\alpha_j}$ and $b_i:=x_{\beta_i}$.
Fix arbitrary $(i,j)\in[\mu]^2$ and $(\alpha,\beta)\in a_j\times b_i$. Then:
$$\eta+1<\alpha^*\le\max\{\alpha^*,\lambda(\beta^*,\beta),\lambda(\delta_j,\beta^*)\}\le\varsigma_j\le\alpha_j<\alpha<\delta_j<\beta^*<\beta_i<\beta.$$
In particular, Fact~\ref{fact2} yields the following conclusions:
\begin{itemize}
\item[(a)] from $\lambda(\beta^*,\beta)<\alpha<\beta^*<\beta$, we have $\tr(\alpha,\beta)=\tr(\beta^*,\beta){}^\smallfrown\tr(\alpha,\beta^*)$;
\item[(b)] from $\lambda(\delta_j,\beta^*)<\alpha<\delta_j<\beta^*$, we have $\tr(\alpha,\beta^*)=\tr(\delta_j,\beta^*){}^\smallfrown\tr(\alpha,\delta_j)$.
\end{itemize}
Altogether,
$$\tr(\alpha,\beta)=\tr(\beta_i,\beta){}^\smallfrown\tr(\beta^*,\beta_i){}^\smallfrown\tr(\delta_j,\beta^*){}^\smallfrown \tr(\alpha,\delta_j).$$
 
For each $i<\mu$, set $u_i:=\{ \tr_h(\beta^*,\beta)\mid \beta\in b_i\}$.
As $\beta_i\in \Gamma_i$, Clause~(1) above implies that $i\in\im(\varrho)$ for all $\varrho\in u_i$.
For each $j<\mu$, set $v_j:=\{ \tr_h(\alpha,\beta^*)\mid\alpha\in a_j\}$
and $\sigma_j:=\tr_h(\delta_j,\beta^*)$.
As $\delta_j\in H_j$, we infer that  $\sigma_j{}^\smallfrown\langle j\rangle\sqsubseteq\sigma$ for all $\sigma\in v_j$.

Next, by the choice of $d_1$, fix $(i,j)\in[\mu]^2$ 
such that $d_1(\varrho{}^\smallfrown\sigma)=(\ell(\varrho),\tau^*,\xi,\phi)$ for all $\varrho\in u_i$ and $\sigma\in v_j$.
Set $a:=a_j$ and $b:=b_i$, so that $(a,b)\in[\mathcal A]^2$.

To see that $\mathbf t[a\times b]=\{(\tau^*,\alpha^*,\beta^*)\}$, fix arbitrary $\alpha\in a$ and $\beta\in b$.
Denote $\varrho:=\tr_h(\beta^*,\beta)$ and $\sigma:=\tr_h(\alpha,\beta^*)$, so that $\varrho\in u_i$ and $\sigma\in v_j$.
Then $d_1(\tr_h(\alpha,\beta))=(\ell(\varrho),\tau^*,\xi,\phi)$, so that
\begin{itemize}
\item $\Tr(\alpha,\beta)(\ell(\varrho))=\Tr(\alpha,\beta)(\rho_2(\beta^*,\beta))=\beta^*$;
\item $\eta=\psi_{\beta^*}(\phi)$ and $\eta+1<\alpha$;
\item $\tau^*<\alpha^*$.
\end{itemize}

Now, since $\lambda(\alpha^*,\alpha)=\eta<\eta+1<\alpha^*<\alpha$,
$\tr(\eta+1,\alpha)=\tr(\alpha^*,\alpha){}^\smallfrown\tr(\eta+1,\alpha^*)$.
So, since $\rho_2(\alpha^*,\alpha)=\xi^{\alpha^*,\alpha}$,
 $\rho_2(\alpha^*,\alpha)=\xi^{\eta+1,\alpha}$ and 
$\alpha^*=\Tr(\eta+1,\alpha)(\xi^{\eta+1,\alpha})$.
\end{proof}

\section{Clause~(1) of Theorem~C}\label{Squaresection}
In this section, we suppose that $\square(\kappa)$ holds.
Fix an arbitrary $\chi\in\reg(\kappa)$ with $\chi^{+}<\kappa$.
We shall construct a witness to $\pl_1(\kappa,\kappa,\chi)$.
Denote $\mu:=\chi^+$.

\begin{lemma}\label{squarelemma} There exists a $C$-sequence $\vec C=\langle C_\alpha\mid\alpha<\kappa\rangle$ satisfying the following:
\begin{enumerate}
\item $C_{\alpha+1}=\{0,\alpha\}$ for every $\alpha<\kappa$;
\item for every club $D\s\kappa$, there exists $\delta\in E^\kappa_{\neq\mu}$ with $\sup(\nacc(C_\delta)\cap D)=\delta$;
\item for every $\alpha\in\acc(\kappa)$ and $\bar\alpha\in\acc(C_\alpha)$, $C_{\bar\alpha}=C_\alpha\cap\bar\alpha$;
\item for every $\gamma<\kappa$, $\{ \delta\in E^\kappa_\chi\mid \min(C_\delta)=\gamma\}$ is stationary.
\end{enumerate}
\end{lemma}
\begin{proof} As $\square(\kappa)$ holds, we may appeal to \cite[Proposition~3.5]{paper24} with $S:=E^\kappa_{\neq\mu}$,
and obtain a $C$-sequence $\vec C$ satisfying Clauses (2) and (3). In particular, $\vec C$ is a $\square(\kappa)$-sequence.
Now, by feeding $\Gamma:=E^\kappa_\chi$ and $\vec C$ to the proof of \cite[Proposition~3.2]{paper18},
we obtain a $C$-sequence $\langle \bar C_\alpha\mid\alpha<\kappa\rangle$ satisfying Clauses (1), (3) and (4).
An inspection of the said proof makes clear that $\sup(\bar C_\alpha\symdiff C_\alpha)<\alpha$
for every $\alpha\in\acc(\kappa)$, so that Clause~(2) is valid for $\langle \bar C_\alpha\mid\alpha<\kappa\rangle$, as well.
\end{proof}
Let $\vec C$ be given by the preceding lemma.
Recalling Subsection~\ref{subsectionwalks}, we now let $\Tr,\tr,\lambda$ and $\rho_2$ be the characteristic functions of walking along $\vec C$,
and let $\eta_{\alpha,\beta}$ be the notation established in Definition~\ref{etanotation}.

Fix a bijection $\pi:\kappa\leftrightarrow\kappa\times\kappa$.
Define a function $g:\kappa\rightarrow\kappa\times\kappa$ via $g(\alpha):=\pi(\min(C_\alpha))$.
Define a function $h:\kappa\rightarrow\mu$ by letting $h(\alpha):=\min(C_\alpha)$ for all $\alpha<\kappa$ with $\min(C_\alpha)<\mu$,
and $h(\alpha):=0$, otherwise.
Then, define a function $\tr_h:[\kappa]^2\rightarrow{}^{<\omega}\mu$ via $\tr_h(\alpha,\beta):=h\circ\tr(\alpha,\beta)$.
Also, for each $(\eta,\tau)\in\kappa\times\kappa$,
denote $G_{\eta,\tau}:=\{\delta<\kappa\mid g(\delta)=(\eta,\tau)\}$,
and for each $i<\mu$, denote $H_i:=h^{-1}\{i\}$.

\begin{lemma}\label{squarelast} For every $(\delta,\beta)\in[\kappa]^2$, $C_\delta=C_{\last{\delta}{\beta}}\cap\delta$.
In particular:
\begin{itemize}
\item $h(\delta)=h(\last{\delta}{\beta})$;
\item for every $\epsilon<\delta$, $\lambda(\epsilon,\delta)=\lambda(\epsilon,\last{\delta}{\beta})$.
\end{itemize}
\end{lemma}
\begin{proof} By Lemma~\ref{squarelemma}(3) together with Lemma~\ref{last}(2).
\end{proof}

Exactly as in Subsection~\ref{ssectioncase1},
we appeal to Lemma~\ref{lemma42} to fix a map $d_0:{}^{<\omega}\mu\rightarrow\omega$, its corresponding $\mu$-additive proper ideal $J$,
define a coloring $c:[\kappa]^2\rightarrow\kappa\times\kappa$ via 
$$c(\alpha,\beta):=g(\Tr(\alpha,\beta)(d_0(\tr_h(\alpha,\beta)))),$$
and define the sets $S_\eta$ and the transformation $\mathbf{t}$ in the very same way.

\begin{defn}  For $\eta<\kappa$, $S_{\eta}$ denotes the set of all $\epsilon<\kappa$ 
with the property that, for every $\varsigma<\kappa$, there exist $I\in J^+$ 
and a sequence $\langle \beta_i\mid i\in I\rangle\in\prod_{i\in I}H_i\setminus\varsigma$, such that,
for all $i\in I$ and $\beta\in x_{\beta_i}$:
\begin{enumerate}
\item[(i)]  $i\in\im(\tr_h(\epsilon,\beta))$;
\item[(ii)] $\lambda(\epsilon,\beta)=\eta$;
\item[(iii)]  $\rho_2(\epsilon,\beta)=\eta_{\epsilon,\beta}$.
\end{enumerate}
\end{defn}
\begin{defn}
Define $\mathbf t:[\kappa]^2\rightarrow[\kappa]^3$
by letting, for all $(\alpha,\beta)\in[\kappa]^2$, $\mathbf t(\alpha,\beta):=(\tau,\alpha^*,\beta^*)$ provided that the following conditions are met:
\begin{itemize}
\item $(\eta,\tau):=c(\alpha,\beta)$ and $\max\{\eta+1,\tau\}<\alpha$,
\item $\beta^*=\Tr(\alpha,\beta)(\eta_{\alpha,\beta})$ is $>\alpha$, and
\item $\alpha^*=\Tr(\eta+1,\alpha)(\eta_{\eta+1,\alpha})$.
\end{itemize}
Otherwise, let $\mathbf t(\alpha,\beta):=(0,\alpha,\beta)$.
\end{defn}

To verify that $\mathbf t$ witnesses $\pl_1(\kappa,\kappa,\chi)$,
suppose that we are given a family $\mathcal A\s[\kappa]^{<\chi}$ consisting of $\kappa$ many pairwise disjoint sets.
Fix a sequence  $\vec x=\langle x_\delta\mid \delta<\kappa\rangle$ such that,
for all $\delta<\kappa$, $x_\delta\in\mathcal A$ with $\min(x_\delta)>\delta$.

\begin{lemma}  There exists $\eta<\kappa$ for which $S_\eta$ is stationary.
\end{lemma}
\begin{proof} It suffices to prove that, for every club $D\s\kappa$,
there exist $\epsilon\in D$ and $\eta<\epsilon$ for which $\epsilon\in S_\eta$.
Thus, let $D$ be an arbitrary club in $\kappa$.

Define a function $f:E^\kappa_\chi\rightarrow\kappa$ via
$$f(\delta):=\sup\{\lambda(\last{\delta}{\beta},\beta)\mid \beta\in x_\delta\}.$$

As $|x_\delta|<\chi=\cf(\delta)$, Lemma~\ref{last}(1) entails that $f$ is regressive.
So, for all $i<\mu$, let us pick a stationary subset $\bar H_i\subseteq H_i$ such that $f\restriction \bar H_i$ is constant.
Set $\zeta:=\sup(f[\bigcup_{i<\mu}\bar H_i])$.
Now, by Lemma~\ref{squarelemma}(2), let us pick a nonzero $\gamma\in E^\kappa_{\neq\mu}$  with $\sup(\nacc(C_\gamma)\cap (D\setminus\zeta))=\gamma$.

Let $\varsigma<\kappa$. Fix a sequence $\langle \beta^\varsigma_i\mid i<\mu\rangle\in\prod_{i<\mu}\bar H_i\setminus\max\{\gamma+1,\varsigma\}$.
For every $i<\mu$, let
$$\zeta^\varsigma_i:=\begin{cases}
0,&\text{if }\gamma\in\acc(C_{\beta^\varsigma_i});\\
\sup(C_{\beta^\varsigma_i}\cap\gamma),&\text{if }\gamma\in\nacc(C_{\beta^\varsigma_i});\\
\lambda(\last{\gamma}{\beta_i^\varsigma},\beta_i^\varsigma),&\text{otherwise}.
\end{cases}$$
Note that, by Lemma~\ref{last}(1), $\zeta^\varsigma_i<\gamma$.

As $\cf(\gamma)\neq\mu$ and as $J$ is a $\mu$-additive proper ideal on $\mu$, we may now fix $I^\varsigma\in J^+$ along with some ordinal $\xi^\varsigma<\gamma$ such that $\max\{\zeta,\zeta_i^\varsigma\}\le\xi^\varsigma$ for all $i\in I^\varsigma$.
Then, pick a large enough $\epsilon^\varsigma\in\nacc(C_\gamma)\cap D$ such that $\sup(C_\gamma\cap\epsilon^\varsigma)>\xi^\varsigma$.

Fix $\epsilon\in\nacc(C_\gamma)\cap D$ for which $\Sigma:=\{\varsigma<\kappa\mid \epsilon^\varsigma=\epsilon\}$
is cofinal in $\kappa$. Denote $\eta:=\sup(C_\gamma\cap\epsilon)$, so that $\eta<\epsilon$.
We have $\epsilon\in D$. To see that $\epsilon\in S_\eta$, let $\varsigma<\kappa$ be arbitrary. 
By increasing $\varsigma$, we may assume that $\varsigma\in\Sigma$.
Let $i\in I^\varsigma$ and $\beta\in x_{\beta^\varsigma_i}$ be arbitrary. We must show that:
\begin{enumerate}
\item[(i)] $i\in\im(\tr_h(\epsilon,\beta))$;
\item[(ii)] $\lambda(\epsilon,\beta)=\eta$;
\item[(iii)]  $\rho_2(\epsilon,\beta)=\eta_{\epsilon,\beta}$.
\end{enumerate}

We have:
$$\lambda(\last{\beta_i^\varsigma}{\beta},\beta)\le f(\beta_i^\varsigma)\le\zeta\le\xi^\varsigma<\eta<\epsilon<\gamma<\beta_i^\varsigma<\beta.$$
It thus follows from Fact~\ref{fact2} that $\tr(\epsilon,\beta)=\tr(\last{\beta_i^\varsigma}{\beta},\beta){}^\smallfrown\tr(\epsilon,\last{\beta_i^\varsigma}{\beta})$.
So, since $\beta_i^\varsigma\in H_i$, Lemma~\ref{squarelast} implies that $i\in\im(\tr_h(\epsilon,\beta))$. 

\begin{claim} $\lambda(\epsilon,\beta)=\eta$ and $\rho_2(\epsilon,\beta)=\eta_{\epsilon,\beta}$.
\end{claim}
\begin{proof} 
By Lemma~\ref{fact3}, $\lambda(\epsilon,\beta)=\max\{\lambda(\last{\beta_i^\varsigma}{\beta},\beta),\lambda(\epsilon,\last{\beta_i^\varsigma}{\beta})\}$.
Now, there are three cases to consider:

$\br$ If $\gamma\in\acc(C_{\beta_i^\varsigma})$, then $C_{\beta_i^\varsigma}\cap\gamma=C_\gamma$,
and since $\epsilon\in C_\gamma$,
$\tr(\epsilon,\beta)=\tr(\last{\beta_i^\varsigma}{\beta},\beta){}^\smallfrown\langle \last{\beta_i^\varsigma}{\beta}\rangle$,
and $\lambda(\epsilon,\last{\beta_i^\varsigma}{\beta})=\sup(C_\gamma\cap\epsilon)=\eta>\zeta\ge\lambda(\last{\beta_i^\varsigma}{\beta},\beta)$,
so the conclusion follows.

$\br$ If $\gamma\in\nacc(C_{\beta_i^\varsigma})$, then since $\epsilon\in C_\gamma$,
$\tr(\epsilon,\beta)=\tr(\last{\beta_i^\varsigma}{\beta},\beta){}^\smallfrown\langle \last{\beta_i^\varsigma}{\beta},\gamma\rangle$,
so that $\lambda(\epsilon,\beta)=\max\{\lambda(\last{\beta_i^\varsigma}{\beta},\beta),\sup(C_{\last{\beta_i^\varsigma}{\beta}}\cap\epsilon),\sup(C_\gamma\cap\epsilon)\}=
\max\{\lambda(\last{\beta_i^\varsigma}{\beta},\beta),\zeta^\varsigma_i,\eta\}$,
and the conclusion follows.

$\br$ Otherwise, $\last{\gamma}{\beta_i^\varsigma}\neq\beta_i^\varsigma$.
Then $\lambda(\last{\gamma}{\beta_i^\varsigma},\beta_i^\varsigma)=\zeta^\varsigma_i\le\xi^\varsigma<\epsilon<\gamma\le\last{\gamma}{\beta_i^\varsigma}<\beta^\varsigma_i$,
and so, by Fact~\ref{fact2},
$\tr(\epsilon,\beta_i^\varsigma)=\tr(\last{\gamma}{\beta_i^\varsigma},\beta_i^\varsigma){}^\smallfrown\tr(\epsilon,\last{\gamma}{\beta_i^\varsigma})$.
Thus,  by Lemma~\ref{fact3},
$$\lambda(\epsilon,\beta_i^\varsigma)=\max\{\lambda(\last{\gamma}{\beta_i^\varsigma},\beta_i^\varsigma),\lambda(\epsilon,\last{\gamma}{\beta_i^\varsigma})\}=\max\{\zeta_i^\varsigma,\lambda(\epsilon,\last{\gamma}{\beta_i^\varsigma})\}.$$

By Lemma~\ref{squarelast}, $\lambda(\epsilon,\last{\beta_i^\varsigma}{\beta})=\lambda(\epsilon,\beta_i^\varsigma)$.
As $\epsilon\in C_\gamma=C_{\last{\gamma}{\beta_i^\varsigma}}\cap\gamma$, we get that $\lambda(\epsilon,\last{\gamma}{\beta_i^\varsigma})=\sup(C_\gamma\cap\epsilon)=\eta$.
Altogether, $\lambda(\epsilon,\beta)=\max\{\lambda(\last{\beta_i^\varsigma}{\beta},\beta),\zeta_i^\varsigma,\eta\}$.
But, $\eta>\xi^\varsigma\ge\max\{\zeta,\zeta_i^\varsigma\}\ge\{\lambda(\last{\beta_i^\varsigma}{\beta},\beta),\zeta_i^\varsigma\}$,
and the conclusion follows.

\end{proof}
This completes the proof.
\end{proof}

Let $\eta$ be given by the preceding lemma.
Let $D$ be a club in $\kappa$ such that, for all $\delta\in D$, there exists $M_\delta\prec\mathcal H_{\kappa^{+}}$ containing 
the parameter $p:=\{S_\eta,\vec x,\vec C,h\}$
and satisfying $M_\delta\cap\kappa=\delta$. 
Consider the club $$E:=\diagonal_{\tau<\kappa}\acc^+(G_{\eta,\tau}\cap\bigcap\nolimits_{j<\mu}\acc^+(H_j\cap D)).$$
Finally, let $S^*:=\{\epsilon\in S_{\eta} \mid \sup(E\cap\epsilon\setminus C_\epsilon)=\epsilon\}.$
\begin{lemma} $S^*$ is stationary.
\end{lemma}
\begin{proof} As $\vec C$ is a $\square(\kappa)$-sequence, \cite[Lemma~1.23]{paper29} implies that $\vec C$ is \emph{amenable}
in the sense of \cite[Definition~1.3]{paper29}, so that
$\{\epsilon\in\kappa\mid \sup(E\cap\epsilon\setminus C_\epsilon)<\epsilon\}$ is nonstationary.
\end{proof}
\begin{lemma}
Let $(\tau^*,\alpha^*,\beta^*)\in\kappa\circledast S^*\circledast S^*$. There exists $(a,b)\in[\mathcal A]^2$ such that $\mathbf t[a\times b]=\{(\tau^*,\alpha^*,\beta^*)\}$.
\end{lemma}
\begin{proof} As $\beta^*\in S^*\s S_{\eta}$, let us pick $I\in J^+$ and a sequence $\langle \beta_i\mid i\in I\rangle\in\prod_{i\in I}H_i\setminus(\beta^*+1)$ such that, for all $i\in I$ and $\beta\in x_{\beta_i}$:
\begin{enumerate}
\item $i\in\tr_h(\beta^*,\beta)$;
\item $\lambda(\beta^*,\beta)=\eta$;
\item  $\rho_2(\beta^*,\beta)=\eta_{\beta^*,\beta}$.
\end{enumerate}

Denote $G:=G_{\eta,\tau^*}\cap\bigcap\nolimits_{j<\mu}\acc^+(H_j\cap D)$.
From $\beta^*\in S^*$ and as $C_{\beta^*}$ is closed, it follows that $\sup(G\cap\beta^*\setminus C_{\beta^*})=\beta^*$. 
Thus, we pick a large enough $\gamma\in G\cap \beta^*\setminus C_{\beta^*}$ such that $\sup(C_{\beta^*}\cap\gamma)>\alpha^*$. 
In particular, for $\varepsilon:=\last{\gamma}{\beta^*}$,  $\lambda(\varepsilon,\beta^*)>\alpha^*>\eta$.

For each $j<\mu$, as $\gamma\in G\s \acc^+(H_j\cap D)$,
Lemma~\ref{last}(1) entails that we may pick a large enough $\delta_j\in H_j\cap D\cap\gamma$ such that $\delta_j>\lambda(\varepsilon,\beta^*)$.
As $M_{\delta_j}$ contains $p$, we have that $S_{\eta}\in M_{\delta_j}$.
By Lemma~\ref{last}(1), $\varsigma_j:=\max\{\alpha^*,\lambda(\varepsilon,\beta^*),\lambda(\last{\delta_j}{\gamma},\gamma)\}+1$ is smaller than $\delta_j$.\footnote{By Convention~\ref{conv28},
if $\last{\delta_j}{\gamma}=\gamma$, then $\lambda(\last{\delta_j}{\gamma},\gamma)=0$.}
Since $\alpha^*\in M_{\delta_j}\cap S_{\eta}$, we may then find $\alpha_j\in M_{\delta_j}\cap(\bigcup_{i<\mu}H_i)\setminus\varsigma_j$ such that, for all $\alpha\in x_{\alpha_j}$:
\begin{enumerate}
\item[(2')] $\lambda(\alpha^*,\alpha)=\eta$;
\item[(3')] $\rho_2(\alpha^*,\alpha)=\eta_{\alpha^*,\alpha}$.
\end{enumerate}

Note that from $\alpha_j\in M_{\delta_j}$, it follows that $\sup(x_{\alpha_j})<\delta_j$. 
Write  $a_j:=x_{\alpha_j}$ and $b_i:=x_{\beta_i}$.
Let $(i,j,\alpha,\beta)\in I\times\mu\times a_j\times b_i$ be arbitrary. Then:
$$\eta<\eta+1<\alpha^*<\varsigma_j\le\alpha_j<\alpha<\delta_j<\gamma\le\varepsilon<\beta^*<\beta_i<\beta.$$
In particular, Fact~\ref{fact2} yields the following conclusions:
\begin{itemize}
\item[(a)] from $\lambda(\beta^*,\beta)=\eta<\alpha<\beta^*<\beta$, we have $\tr(\alpha,\beta)=\tr(\beta^*,\beta){}^\smallfrown\tr(\alpha,\beta^*)$;
\item[(b)] from $\lambda(\varepsilon,\beta^*)<\varsigma_j<\alpha<\beta^*$, we have $\tr(\alpha,\beta^*)=\tr(\varepsilon,\beta^*){}^\smallfrown\tr(\alpha,\varepsilon)$;
\item[(c)] from $\lambda(\last{\delta_j}{\gamma},\varepsilon)=\lambda(\last{\delta_j}{\gamma},\gamma)<\varsigma_j<\alpha<\delta_j\le\last{\delta_j}{\gamma}\le\gamma\le\varepsilon$,
we have $\tr(\alpha,\varepsilon)=\tr(\last{\delta_j}{\gamma},\varepsilon){}^\smallfrown \tr(\alpha,\last{\delta_j}{\gamma})$.
\end{itemize}
So that, altogether,
$$\tr(\alpha,\beta)=\tr(\beta^*,\beta){}^\smallfrown\tr(\varepsilon,\beta^*){}^\smallfrown\tr(\last{\delta_j}{\gamma},\varepsilon){}^\smallfrown \tr(\alpha,\last{\delta_j}{\gamma}).$$

In addition, from $\lambda(\alpha^*,\alpha)=\eta<\eta+1<\alpha^*<\alpha$, we infer that 
\begin{itemize}
\item[(d)] $\tr(\eta+1,\alpha)=\tr(\alpha^*,\alpha){}^\smallfrown\tr(\eta+1,\alpha^*)$.
\end{itemize}

For each $i\in I$, denote $u_i:=\{ \tr_h(\varepsilon,\beta)\mid \beta\in b_i\}$.
By Clause~(1) above, for all $\varrho\in u_i$, $i\in\im(\tr_h(\beta^*,\beta))\s\im(\varrho)$.

For each $j<\mu$, denote $v_j:=\{ \tr_h(\alpha,\varepsilon)\mid\alpha\in a_j\}$.
By Clause~(c) above, for all $\sigma\in v_j$, $\tr_h(\last{\delta_j}{\gamma},\varepsilon){}^\smallfrown\langle j\rangle\sqsubseteq\sigma$.

Next, by the choice of $d_0$, fix $(i,j)\in[I]^2$ 
such that $d_0(\varrho{}^\smallfrown\sigma)=\ell(\varrho)$ for all $\varrho\in u_i$ and $\sigma\in v_j$.
Set $a:=x_j$ and $b:=x_i$. The rest of the proof is now identical to that of Lemma~\ref{wrappingup}.
\end{proof}

\section{Clause~(3) of Theorem~C}\label{lastclause}
In this section, we suppose that $\kappa$ is inaccessible, $\chi\in\reg(\kappa)$,
and $E^\kappa_{\ge\chi}$ admits a stationary set that does not reflect at inaccessibles.
Let $\mu:=\chi^+$. We shall prove that $\pl_1(\kappa,\mu,\chi)$ holds. 
Note that by the result of Section~\ref{mainthm5},
we may assume that every stationary subset of $E^\kappa_{\ge\chi}$ reflects.

\begin{lemma}  There exist $\sigma^1,{\sigma^0}\in\reg(\kappa)$ with $\mu<\sigma^1<{\sigma^0}$
and stationary subsets $S^1,S^0$ of $\kappa$ consisting of singular cardinals such that
\begin{itemize}
\item $S^1\s E^\kappa_{\sigma^1}$, and $S^1$ does not reflect at inaccessibles;
\item $S^0\s E^\kappa_{\sigma^0}$, and $S^0$ does not reflect at inaccessibles.
\end{itemize}
\end{lemma}
\begin{proof} Fix a stationary subset $T\s E^{\kappa}_{\ge\chi}$ that does not reflect at inaccessibles.
Since $\card(\kappa)$ is a club in the inaccessible $\kappa$, we may assume that $T\s\card(\kappa)$,
so that  $\Tr(T)$ is a stationary set consisting of singular cardinals.
By Fodor's lemma, fix a cardinal $\nu\in\reg(\kappa)\setminus\mu$ for which
$R:=\Tr(T)\cap E^\kappa_{\nu}$ is stationary.
As $\Tr(R) \s \Tr(T)$, we can repeat the process to find $\sigma^1\in \reg(\kappa)\setminus(\nu+1)$ such that $\Tr(R) \cap E^\kappa_{\sigma^1}$ is stationary.
Now $S^1:=\Tr(R)\cap E^\kappa_{\sigma^1}\setminus\{\sigma^1\}$ is a stationary set consisting of singular cardinals.
Repeating the process for the last time, we find ${\sigma^0} \in \reg(\kappa)\setminus(\sigma^1+1)$ such that $S^0 := \Tr(S) \cap E^\kappa_{\sigma^0}\setminus\{\sigma^0\}$ is stationary.
Then ${\sigma^0}>\sigma^1>\nu\ge\mu$ and $\Tr(S^0)\s\Tr(S^1)\s\Tr(T)$, so $\sigma^1$, ${\sigma^0}$, $S^1$, and $S^0$ are as sought.
\end{proof}

Let $\sigma^1$, ${\sigma^0}$, $S^1$, and $S^0$ be given by the preceding claim.
Note that since $S^1$ consists of singular cardinals, $\min(S^1)>\sigma_1$.
By \cite[Theorem~2.1.1]{MR3321938}, we fix a sequence $\vec e=\langle e_\delta\mid\delta\in S^1\rangle$ such that
  \begin{itemize}
    \item for all $\delta\in S^1$, $e_\delta$ is a club in $\delta$ of order type $\sigma^1$;
    \item for all $\delta\in S^1$, $\langle \cf(\gamma)\mid \gamma\in\nacc(e_\delta)\rangle$
      is strictly increasing, converging to $\delta$;
    \item for every club $D\s\kappa$, there exists $\delta\in S^1$ with $e_\delta\s D$.
  \end{itemize}

\begin{lemma}\label{swallow} There exists a $C$-sequence
  $\vec C=\langle C_\alpha\mid\alpha<\kappa\rangle$ such that, for all $\alpha<\kappa$:
  \begin{enumerate}
\item $|C_\alpha|=\cf(\alpha)$;
\item if $\acc(C_\alpha)\cap S^1\neq\emptyset$, then $\min(C_\alpha)\ge\cf(\alpha)>\sigma^1$;
\item for every $\delta\in(\acc(C_\alpha)\cup\{\alpha\})\cap S^1$, $\sup(e_\delta\setminus C_\alpha)<\delta$.
  \end{enumerate}
\end{lemma}
\begin{proof} 
This is a standard club-swallowing trick, but we do not know of a reference in which the above precise properties are exposed.

By recursion on $n<\omega$, we shall define a $C$-sequence $\vec{C^n}=\langle C_\alpha^n\mid \alpha<\kappa\rangle$, as follows.
We commence with the case $n=0$:

$\br$ Let $C_0^0:=\emptyset$ and $C_{\alpha+1}^0:=\{\alpha\}$ for all $\alpha<\kappa$.

$\br$ For each $\alpha\in\acc(\kappa)\setminus(\reg(\kappa)\cup S^1)$, let $C^0_\alpha$ be a club in $\alpha$
with $\otp(C^0_\alpha)=\cf(\alpha)=\min(C^0_\alpha)$.

$\br$ For each $\alpha\in S^1$, let $C_\alpha^0:=e_\alpha\setminus\cf(\alpha)$.

$\br$ For each $\alpha\in\reg(\kappa)$, since $S^1$ consists of singular cardinals and does not reflect at inaccessibles,
we may let $C_\alpha^0$ be a club in $\alpha$ with $\acc(C_\alpha)\cap S^1=\emptyset$.

Next, suppose that $n<\omega$ is such that $\vec{C^n}$ has already been defined
to satisfy requirements (1) and (2) of the lemma.
Define a $C$-sequence $\vec C^{n+1}=\langle C_\alpha^{n+1}\mid \alpha<\kappa\rangle$ by letting,
for each $\alpha<\kappa$, $C_\alpha^{n+1}$ be the closure in $\alpha$ of the set
$$C_\alpha^n\cup\bigcup\{ e_\delta\setminus\cf(\alpha)\mid \delta\in\acc(C_\alpha^n)\cap S^1\}.$$

To see that Clauses (1) and (2) remain valid also for $\vec C^{n+1}$, let $\alpha<\kappa$ be arbitrary.
If $C^n_\alpha=C^{n+1}_\alpha$, then we are done, 
so assume $C^n_\alpha\neq C^{n+1}_\alpha$.
In particular, $\acc(C^n_\alpha)\cap S^1\neq\emptyset$,
so that, by the inductive hypothesis, $|C^n_\alpha|=\cf(\alpha)>\sigma^1=\cf(\delta)$ for all $\delta\in \acc(C^n_\alpha)\cap S^1$.
In effect, $|C^{n+1}_\alpha|=\cf(\alpha)$.

Finally, for each $\alpha<\kappa$, let $C_\alpha$ be the closure in $\alpha$ of $\bigcup_{n<\omega}C_\alpha^n$.
As $S^1\s E^\kappa_{\sigma^1}\s E^\kappa_{>\omega}$, the above construction ensures that Clause~(3) holds, as well.
\end{proof}

Let $\vec C$ be given by the preceding lemma.
Recalling Subsection~\ref{subsectionwalks}, we now let $\Tr,\tr,\lambda$ and $\rho_2$ be the characteristic functions of walking along $\vec C$,
and let $\eta_{\alpha,\beta}$ be the notation established in Definition~\ref{etanotation}.

\begin{defn} For every  $(\delta,\beta)\in S^1\circledast\kappa$, let $\Lambda(\delta,\beta)$ denote the least $\gamma\in\nacc(e_\delta)$ such that all of the following hold:
\begin{itemize}
\item $\gamma>\lambda(\last{\delta}{\beta},\beta)$;
\item $\cf(\gamma)>\cf(\last{\delta}{\beta})$;
\item $e_\delta\setminus\sup(e_\delta\cap\gamma)\s C_{\last{\delta}{\beta}}$.
\end{itemize}
\end{defn}

\begin{lemma}\label{nacclemma} 
Let $(\delta,\beta)\in S^1\circledast\kappa$.
Then $\Lambda(\delta,\beta)$ is well-defined, and:
\begin{enumerate}
\item $\nacc(e_\delta)\setminus\Lambda(\delta,\beta)\s\nacc(C_{\last{\delta}{\beta}})$;
\item for every $\varepsilon\in \nacc(C_{\last{\delta}{\beta}})\cap[\Lambda(\delta,\beta),\delta)$, $\sup(e_\delta\cap\varepsilon)\le\lambda(\varepsilon,\beta)<\varepsilon$;
\item for every $\varepsilon\in \nacc(C_{\last{\delta}{\beta}})\cap[\Lambda(\delta,\beta),\delta)$, $\min(\im(\tr(\varepsilon,\beta))=\last{\delta}{\beta}$;
\item $\cf({\last{\delta}{\beta}})\ge\sigma^1$.
\end{enumerate}
\end{lemma}
\begin{proof} Since $\langle \cf(\gamma)\mid \gamma\in\nacc(e_\delta)\rangle$ is strictly increasing and converging to $\delta$,
the first part of the following claim implies that $\Lambda(\delta,\beta)$ is well-defined.

\begin{claim} $\max\{\lambda(\last{\delta}{\beta},\beta),\cf({\last{\delta}{\beta}}),\sup(e_\delta\setminus C_{\last{\delta}{\beta}})\}<\delta$
and $\cf({\last{\delta}{\beta}})\ge\sigma^1$.
\end{claim}
\begin{proof} 
 By Lemma~\ref{last}(1), $\lambda(\last{\delta}{\beta},\beta)<\delta$. 
Now, there are two cases to consider:

$\br$ If $\last{\delta}{\beta}=\delta$,
then from $\delta\in S^1\s E^\kappa_{\sigma^1}$ and $\min(S^1)>\sigma^1$, we infer that $\cf(\delta)=\sigma^1<\delta$.
Now, by Lemma~\ref{swallow}(3), $\sup(e_\delta\setminus C_\delta)<\delta$. 

$\br$ If $\last{\delta}{\beta}\neq\delta$, then set $\alpha:=\last{\delta}{\beta}$.
By Lemma~\ref{last}(2), $\delta\in\acc(C_{\alpha})$.
So, by Lemma~\ref{swallow}(2), $\delta>\min(C_\alpha)\ge\cf(\alpha)>\sigma^1$.
In addition, by Lemma~\ref{swallow}(3), $\sup(e_\delta\setminus C_\alpha)<\delta$.
\end{proof}

For every $\varepsilon\in\nacc(e_\delta)$ above $\sup(e_\delta\setminus C_{\last{\beta}{\delta}})$ 
and of cofinality greater than $\cf(\last{\beta}{\delta})=|C_{\last{\beta}{\delta}}|$,
we have $\varepsilon\in\nacc(C_{\last{\beta}{\delta}})$, so that Clause~(1) holds.

Now, let $\varepsilon\in \nacc(C_{\last{\delta}{\beta}})\cap[\Lambda(\delta,\beta),\delta)$ be arbitrary.
We have $$\lambda(\last{\delta}{\beta},\beta)<\Lambda(\delta,\beta)\le\varepsilon<\delta\le\last{\delta}{\beta}\le\beta,$$
so, by Fact~\ref{fact2}, $\tr(\varepsilon,\beta)=\tr(\last{\delta}{\beta},\beta){}^\smallfrown\tr(\varepsilon,\last{\delta}{\beta})$
and Clause~(3) holds.
By Lemma~\ref{fact3}, $\lambda(\varepsilon,\beta)=\max\{\lambda(\last{\delta}{\beta},\beta),\sup(C_{\last{\delta}{\beta}}\cap\varepsilon)\}$.
Since $e_\delta\setminus\sup(e_\delta\cap\Lambda(\delta,\beta))\s C_{\last{\delta}{\beta}}$,
we infer that $\sup(C_{\last{\delta}{\beta}}\cap\varepsilon)\ge\sup(e_\delta\cap\varepsilon)$,
and hence Clause~(2) holds as well.
\end{proof}

Define a collection $\mathcal I\s\mathcal P(\kappa)$ via $A\in\mathcal I$ iff
there exists a club $D\s\kappa$ such that for every $\delta\in S^1\cap\acc(D)$, $\sup(\nacc(e_\delta)\cap D\cap A)<\delta$.
It is clear that $\mathcal I$ is a $\sigma^1$-complete ideal over $\kappa$, extending $\ns_{\kappa}$. 
By the choice of $\vec e$, $\mathcal I$ is moreover proper. The next lemma is the only part of the proof that makes use of $S^0$ and $\sigma^0$.

\begin{lemma}  $\mathcal I$ is not weakly $\mu$-saturated, i.e.,
there is a partition $\kappa=\biguplus_{i<\mu}H_i$ such that $H_i\in\mathcal I^+$ for every $i<\mu$.
\end{lemma}
\begin{proof} For each $\delta\in S^1$, let
$I_\delta:=\{ A\s e_\delta\mid \sup(\nacc(e_\delta)\cap A)<\delta\}$, and note:
\begin{itemize}
\item As $\cf(\delta)=\sigma^1$, 
$I_\delta$ is a $\sigma^1$-complete ideal over $e_\delta$;
\item As $\sigma^0$ is a regular cardinal greater than $\cf(\delta)$, 
for every $\s$-increasing sequence $\langle A_j\mid j<\sigma^0\rangle$ of sets from $I_\delta$,
the union $\bigcup_{j<\sigma^0}A_j$ is in $I_\delta$, as well. That is, the ideal $I_\delta$ is $\sigma^0$-indecomposable.
\end{itemize}

Trivially, $\sup_{\delta\in S^1}|e_\delta|^+<\kappa$. Setting $\bar C:=\langle e_\delta\mid\delta\in S^1\rangle$ 
and $\bar I:=\langle I_\delta\mid \delta\in S^1\rangle$,
and recalling \cite[Definition~3.0]{Sh:365}, it is evident that the ideal
$\id_p(\bar C,\bar I)$ is equal to our proper ideal $\mathcal I$.
As $S^0$ is a stationary subset of $E^\kappa_{\sigma^0}$ that does not reflect at
inaccessibles, Case $(\beta)(a)$ of \cite[Claim~3.3]{Sh:365} entails the existence
of a partition of $\kappa$ into $\sigma^0$ many $\mathcal I$-positive sets.
In particular, since $\sigma^0>\mu$, $\mathcal I$ is not weakly $\mu$-saturated.
\end{proof}

By the preceding lemma, fix a surjection $h:\kappa\rightarrow\mu$ such that
  $H_i:=h^{-1}\{i\}$ is in $\mathcal I^+$ for all $i<\mu$. 
Then, define a function $\tr_h:[\kappa]^2\rightarrow{}^{<\omega}\mu$ via $\tr_h(\alpha,\beta):=h\circ\tr(\alpha,\beta)$.

Let $d:{}^{<\omega}\mu\rightarrow\omega\times\mu\times\mu\times\mu$
be the function given by Fact~\ref{pl6} using $\nu:=\chi$.
We are now ready to define our transformation.
\begin{defn}
Define $\mathbf t:[\kappa]^2\rightarrow[\kappa]^3$
by letting, for all $(\alpha,\beta)\in[\kappa]^2$, $\mathbf t(\alpha,\beta):=(\tau^*,\alpha^*,\beta^*)$ provided that,
for $(n,i,j,\tau):=d(\tr_h(\alpha,\beta))$, all of the following conditions are met:
\begin{itemize}
\item $\beta^*=\Tr(\alpha,\beta)(n)$ is $>\alpha$,
\item $\eta:=\lambda(\beta^*,\beta)$ satisfies that $\eta+1<\alpha$,
\item $\alpha^*=\Tr(\eta+1,\alpha)(\eta_{\eta+1,\alpha})$, and
\item $\tau^*=\tau<\alpha^*$.
\end{itemize}
Otherwise, let $\mathbf t(\alpha,\beta):=(0,\alpha,\beta)$.
\end{defn}

To verify that $\mathbf t$ witnesses $\pl_1(\kappa,\mu,\chi)$,
suppose that we are given a family $\mathcal A\s[\kappa]^{<\chi}$ consisting of $\kappa$ many pairwise disjoint sets.

\begin{lemma} For every $i<\mu$, there exist an ordinal $\zeta_i<\kappa$
and a sequence $\langle x_\gamma\mid \gamma\in\bar H_i\rangle$ such that:
\begin{itemize}
\item $\bar H_i$ is a stationary subset of $H_i$;
\item for all $\gamma\in\bar H_i$, $x_\gamma\in\mathcal A$ with $\min(x_\gamma)>\gamma$;
\item for all $\gamma\in\bar H_i$ and $\beta\in x_\gamma$, $\lambda(\gamma,\beta)\le\zeta_i$.
\end{itemize}
\end{lemma}
\begin{proof} Let $i<\mu$.
By the pressing down lemma, it suffices to prove that for every club $D\s\kappa$,
there exist $\gamma\in D\cap H_i$, $\zeta<\gamma$ and $x\in\mathcal A$ with $\min(x)>\gamma$ 
such that  $\lambda(\gamma,\beta)\le\zeta$ for all $\beta\in x$.
Thus, let $D$ be an arbitrary club in $\kappa$.

Since $H_i$ is in $\mathcal I^+$, we may fix $\delta\in S^1$ such that $\sup(\nacc(e_\delta)\cap D\cap H_i)=\delta$. 
Fix any $x\in\mathcal A$ with $\min(x)>\delta$. As $\cf(\delta)=\sigma^1>|x|$,
we may fix a large enough $\varepsilon\in \nacc(e_\delta)\cap D\cap H_i$ above $\sup_{\beta\in x}\Lambda(\delta,\beta)$.
Then, by Clauses (1) and (2) of Lemma~\ref{nacclemma}, $\sup_{\beta\in x}\lambda(\varepsilon,\beta)<\varepsilon$.
So $\gamma:=\varepsilon$ and $\zeta:=\sup_{\beta\in x}\lambda(\gamma,\beta)$ are as sought.
\end{proof}

For each $i<\mu$, let $\zeta_i$ and $\vec{x^i}=\langle x_\gamma\mid \gamma\in\bar H_i\rangle$ be given by the preceding lemma. 
Set $\zeta:=\sup_{i<\mu}\zeta_i$.

\begin{defn} For $\eta<\kappa$, $S_{\eta}$ denotes the set of all $\epsilon<\kappa$ 
with the property that, for every $\varsigma<\kappa$, there exists a sequence $\langle \beta_i\mid i<\mu\rangle\in\prod_{i<\mu}\bar H_i\setminus\varsigma$, such that,
for all $i<\mu$ and $\beta\in x_{\beta_i}$:
\begin{enumerate}
\item[(i)]  $i\in\im(\tr_h(\epsilon,\beta))$;
\item[(ii)] $\lambda(\epsilon,\beta)=\eta$;
\item[(iii)]  $\rho_2(\epsilon,\beta)=\eta_{\epsilon,\beta}$.
\end{enumerate}
\end{defn}

\begin{lemma}  There exists $\eta<\kappa$ for which $S_\eta$ is stationary.
\end{lemma}
\begin{proof} Let $D$ be an arbitrary club in $\kappa$;
we shall find $\epsilon\in D$ and $\eta<\epsilon$ for which $\epsilon\in S_{\eta}$.
By the choice of $\vec e$, the set $\Gamma:=\{\gamma\in S^1\mid \zeta<\gamma\ \&\ e_\gamma\s D\}$ is stationary.
Now, fix $\delta\in S^1$ such that $e_\delta\s\acc^+(\Gamma)$.

Let $\varsigma<\kappa$. Fix any sequence $\langle \beta^\varsigma_i\mid i<\mu\rangle\in\prod_{i<\mu}\bar H_i\setminus\max\{\delta+1,\varsigma\}$.
We shall find an ordinal $\epsilon^\varsigma\in D\cap\delta$, as follows.

As $\cf(\delta)=\sigma^1>\mu$, let us fix a large enough $\varepsilon^\varsigma\in \nacc(e_\delta)$
above $\max\{\zeta,\allowbreak\sup_{i<\mu}\Lambda(\delta,\beta_i^\varsigma)\}$.
As $\langle \cf(\varepsilon)\mid \varepsilon\in\nacc(e_\delta)\rangle$ is strictly increasing and converging to $\delta$,
we may also require that $\cf(\varepsilon^\varsigma)>\mu$.
By Lemma~\ref{nacclemma}(2), $\Lambda^\varsigma:=\max\{\zeta,\sup_{i<\mu}\lambda(\varepsilon^\varsigma,\beta_i^\varsigma)\}$ is smaller than $\varepsilon^\varsigma$.
As $\varepsilon^\varsigma\in\nacc(e_\delta)\s\acc^+(\Gamma)$, let us pick $\gamma^\varsigma\in\Gamma$ with 
$\Lambda^\varsigma<\gamma^\varsigma<\varepsilon^\varsigma$.
Now, fix a large enough $\epsilon^\varsigma\in\nacc(e_{\gamma^\varsigma})\s D\cap\delta$
to satisfy 
$\sup(e_{\gamma^\varsigma}\cap\epsilon^\varsigma)>\max\{\Lambda^\varsigma,\Lambda(\gamma^\varsigma,\varepsilon^\varsigma)\}$.
Denote $\alpha^\varsigma:=\last{\gamma^\varsigma}{\varepsilon^\varsigma}$.

By the pigeonhole principle, let us fix $\epsilon\in D\cap\delta$,
and $\eta\le\epsilon$ for which
$$\Sigma:=\{\varsigma<\kappa\mid \epsilon^\varsigma=\epsilon\ \&\ \sup(C_{\alpha^\varsigma}\cap\epsilon^\varsigma)=\eta\}$$
is cofinal in $\kappa$.
We already know that $\epsilon\in D$; we shall later show that $\eta<\epsilon$.

To see that $\epsilon\in S_{\eta}$, let $\varsigma<\kappa$ be arbitrary. 
By increasing $\varsigma$, we may assume that $\varsigma\in\Sigma$.
Let $i<\mu$ and $\beta\in x_{\beta^\varsigma_i}$ be arbitrary. We shall show that:
\begin{enumerate}
\item[(i')] $\tr(\epsilon,\beta)=\tr(\beta_i^\varsigma,\beta){}^\smallfrown\tr(\epsilon,\beta_i^\varsigma)$;
\item[(ii')] $\lambda(\epsilon,\beta)=\eta$;
\item[(iii')]  $\rho_2(\epsilon,\beta)=\eta_{\epsilon,\beta}$.
\end{enumerate}

We have:
$$\max\{\lambda(\beta_i^\varsigma,\beta),\lambda(\varepsilon^\varsigma,\beta_i^\varsigma)\}\le\max\{\Lambda^\varsigma,\Lambda(\gamma^\varsigma,\varepsilon^\varsigma)\}<\epsilon<\gamma^\varsigma<\varepsilon^\varsigma<\delta<\beta_i^\varsigma<\beta.$$
It thus follows from Fact~\ref{fact2} that Clause~(i') is satisfied,
so that  $i\in\im(\tr_h(\epsilon,\beta))$.
It also follows from Fact~\ref{fact2} that $\tr(\epsilon,\beta_i^\varsigma)=\tr(\varepsilon^\varsigma,\beta_i^\varsigma){}^\smallfrown\tr(\epsilon,\varepsilon^\varsigma)$.
In addition, by Clauses (1) and (3) of Lemma~\ref{nacclemma},
$\tr(\epsilon,\varepsilon^\varsigma)=\tr(\alpha^\varsigma,\varepsilon^\varsigma){}^\smallfrown\tr(\epsilon,\alpha^\varsigma)$.
Thus, altogether:
$$\tr(\epsilon,\beta)=\tr(\beta_i^\varsigma,\beta){}^\smallfrown\tr(\varepsilon^\varsigma,\beta_i^\varsigma){}^\smallfrown\tr(\alpha^\varsigma,\varepsilon^\varsigma){}^\smallfrown\tr(\epsilon,\alpha^\varsigma).$$
As $\epsilon$ is an element of $\nacc(e_{\gamma^\varsigma})$ above $\Lambda(\gamma^\varsigma,\varepsilon^\varsigma)\ge\sup(e_{\gamma^\varsigma}\setminus C_{\alpha^\varsigma})$,
we infer from Lemma~\ref{nacclemma}(1) that $\epsilon\in\nacc(C_{\alpha^\varsigma})$ and hence $\lambda(\epsilon,\alpha^\varsigma)=\sup(C_{\alpha^\varsigma}\cap\epsilon)$.
As $\epsilon=\epsilon^\varsigma$, it follows from Lemma~\ref{nacclemma}(2) that
$$\begin{array}{lll}\max\{\lambda(\beta_i^\varsigma,\beta),\lambda(\varepsilon^\varsigma,\beta_i^\varsigma),\lambda(\alpha^\varsigma,\varepsilon^\varsigma)\}&\le&\max\{\Lambda^\varsigma,\Lambda(\gamma^\varsigma,\varepsilon^\varsigma)\}\\
&<&\sup(e_{\gamma^\varsigma}\cap\epsilon)\\
&\le&\sup(C_{\alpha^\varsigma}\cap\epsilon)\\
&=&\eta.\end{array}$$
Altogether, $\lambda(\epsilon,\beta)=\sup(C_{\alpha^\varsigma}\cap\epsilon)=\eta$ and $\rho_2(\epsilon,\beta)=\eta_{\epsilon,\beta}$.
In addition, since $\eta=\sup(C_{\alpha^\varsigma}\cap\epsilon)$ and $\epsilon\in\nacc(C_{\alpha^\varsigma})$,
we infer that $\eta<\epsilon$, as promised.
\end{proof}

Let $\eta$ be given by the preceding lemma.
Let $D$ be a club in $\kappa$ such that, for all $\delta\in D$, there exists $M_\delta\prec\mathcal H_{\kappa^{+}}$ containing 
the parameter $p:=\{S_\eta,\langle\vec{x^i}\mid i<\mu\rangle,\vec C,h\}$
and satisfying $M_\delta\cap\kappa=\delta$.
For every $j<\mu$, since $H_j$ is in $\mathcal I^+$, 
the set $\Delta_j:=\{\delta\in S^1\mid \sup(\nacc(e_\delta)\cap D\cap H_j)=\delta\}$ is stationary.
Finally, let $$S^*:=S_\eta\cap\bigcap\nolimits_{j<\mu}\acc^+(\Delta_j).$$

\begin{lemma}\label{lemmawithill}
Let $(\tau^*,\alpha^*,\beta^*)\in\mu\circledast S^*\circledast S^*$. There exists $(a,b)\in[\mathcal A]^2$ such that $\mathbf t[a\times b]=\{(\tau^*,\alpha^*,\beta^*)\}$.
\end{lemma}
\begin{proof} As $\beta^*\in S^*\s S_\eta$, let us fix a sequence $\langle \beta_i\mid i<\mu\rangle\in\prod_{i<\mu}\bar H_i\setminus(\beta^*+1)$ 
such that, for all $i<\mu$ and $\beta\in x_{\beta_i}$:
\begin{enumerate}
\item $i\in\im(\tr_h(\beta^*,\beta))$;
\item $\lambda(\beta^*,\beta)=\eta$;
\item $\rho_2(\beta^*,\beta)=\eta_{\beta^*,\beta}$.
\end{enumerate}

For all $j<\mu$, as $\beta^*\in\acc^+(\Delta_j)$, we may pick $\delta_j\in \Delta_j\cap \beta^*$ above $\alpha^*$,
so that $\delta_j>\alpha^*>\eta$.
Now, pick $\varepsilon_j\in\nacc(e_{\delta_j})\cap D\cap H_j$ above $\max\{\alpha^*,\Lambda(\delta_j,\beta^*)\}$.
As $M_{\varepsilon_j}$ contains $p$, we have that $S_\eta\in M_{\varepsilon_j}$.
Now, by Clauses (1) and (2) of Lemma~\ref{nacclemma},
$\varsigma_j:=\max\{\alpha^*,\Lambda(\delta_j,\beta^*),\lambda(\varepsilon_j,\beta^*)\}+1$ is smaller than $\varepsilon_j$.
Since $\alpha^*\in M_{\varepsilon_j}\cap S_\eta$, we may then find $\alpha_j\in M_{\varepsilon_j}\cap\bar H_j\setminus\varsigma_j$ such that, for all $\alpha\in x_{\alpha_j}$:
\begin{enumerate}
\item[(2')] $\lambda(\alpha^*,\alpha)=\eta$;
\item[(3')] $\rho_2(\alpha^*,\alpha)=\eta_{\alpha^*,\alpha}$.
\end{enumerate}

Note that from $\alpha_j\in M_{\varepsilon_j}$, it follows that $\sup(x_{\alpha_j})<\varepsilon_j$. 
Write  $a_j:=x_{\alpha_j}$ and $b_i:=x_{\beta_i}$.
Fix arbitrary $(i,j)\in[\mu]^2$ and $(\alpha,\beta)\in a_j\times b_i$. Then:
$$\eta+1<\alpha^*\le\max\{\alpha^*,\lambda(\beta^*,\beta),\lambda(\varepsilon_j,\beta^*)\}<\varsigma_j\le\alpha_j<\alpha<\varepsilon_j<\beta^*<\beta_i<\beta.$$

So, by Fact~\ref{fact2}:
$$\tr(\alpha,\beta)=\tr(\beta^*,\beta){}^\smallfrown\tr(\varepsilon_j,\beta^*){}^\smallfrown \tr(\alpha,\varepsilon_j).$$

For each $i<\mu$, set $u_i:=\{ \tr_h(\beta^*,\beta)\mid \beta\in b_i\}$.
By Clause~(1) above, $i\in\im(\varrho)$ for all $\varrho\in u_i$.
For each $j<\mu$, set $v_j:=\{ \tr_h(\alpha,\beta^*)\mid\alpha\in a_j\}$
and $\sigma_j:=\tr_h(\varepsilon_j,\beta^*)$.
As $\varepsilon_j\in H_j$, we infer that  $\sigma_j{}^\smallfrown\langle j\rangle\sqsubseteq\sigma$ for all $\sigma\in v_j$.

Finally, by the choice of $d$, fix $(i,j)\in[\mu]^2$ 
such that $d(\varrho{}^\smallfrown\sigma)=(\ell(\varrho),i,j,\tau^*)$ for all $\varrho\in u_i$ and $\sigma\in v_j$.
Set $a:=a_j$ and $b:=b_i$, so that $(a,b)\in[\mathcal A]^2$.

To see that $\mathbf t[a\times b]=\{(\tau^*,\alpha^*,\beta^*)\}$, fix arbitrary $\alpha\in a$ and $\beta\in b$.
Denote $\varrho:=\tr_h(\beta^*,\beta)$ and $\sigma:=\tr_h(\alpha,\beta^*)$, so that $\varrho\in u_i$ and $\sigma\in v_j$.
Denote  $(n,i',j',\tau):=d(\tr_h(\alpha,\beta))$.
Then:
\begin{itemize}
\item $\Tr(\alpha,\beta)(n)=\Tr(\alpha,\beta)(\rho_2(\beta^*,\beta))=\beta^*$;
\item $\tau=\tau^*$;
\item $\eta=\lambda(\beta^*,\beta)$ and $\eta+1<\alpha$;
\item $\tau^*<\mu<\alpha^*$.
\end{itemize}

Now, since $\lambda(\alpha^*,\alpha)=\eta<\eta+1<\alpha^*<\alpha$,
$\tr(\eta+1,\alpha)=\tr(\alpha^*,\alpha){}^\smallfrown\tr(\eta+1,\alpha^*)$.
So, since $\rho_2(\alpha^*,\alpha)=\eta_{\alpha^*,\alpha}$,
 $\rho_2(\alpha^*,\alpha)=\eta_{\eta+1,\alpha}$ and 
$\alpha^*=\Tr(\eta+1,\alpha)(\eta_{\eta+1,\alpha})$.
\end{proof}

\section{Acknowledgements}
The first author is partially supported by the European Research Council (grant agreement ERC-2018-StG 802756) and by the Israel Science Foundation (grant agreement 2066/18).
The second author is supported by the Foreign Postdoctoral Fellowship Program of the Israel Academy of Sciences and Humanities and by the Israel Science Foundation (grant agreement 2066/18).

The results of this paper were presented by the first author in a talk at the Oberseminar Mengenlehre at Universit\"at M\"unster, January 2020. 
He thanks the hosts for the warm hospitality and the participants of the seminar for their feedback.

The authors are grateful to the anonymous referee for a thorough reading of the paper, and for providing a thoughtful list of corrections.

\end{document}